\documentclass[12pt]{amsart}
\RequirePackage{mathtools}
\RequirePackage{amsopn}
\RequirePackage{amsfonts}
\RequirePackage{paralist}
\RequirePackage{amssymb}
\RequirePackage{amsthm}

\usepackage[alphabetic]{amsrefs}

\renewcommand\MR[1]{\relax} 
\usepackage{tikz}
\usetikzlibrary{cd,decorations.pathmorphing}
\usepackage{mathrsfs}
\usepackage{d-bends}
\usepackage[useregional,english]{datetime2}
\usepackage[T1]{fontenc}

\newtheorem{thm}{Theorem}[section]
\numberwithin{equation}{section}

\newtheorem{lemma}[thm]{Lemma}
\newtheorem{prop}[thm]{Proposition}

\theoremstyle{definition}
\newtheorem{definition}[thm]{Definition}

\theoremstyle{remark}
\newtheorem{remark}[thm]{Remark}
\newtheorem{example}[thm]{Example}

\newtheorem{notation}[thm]{Notation}

\newtheorem{mycomment}[thm]{Comment}
{\end{mycomment}\endgroup}
\def\mathcs{C^{*}}
\newcommand{\cs}{\ensuremath{\mathcs}}

\DeclareMathSymbol{\rtimes}{\mathbin}{AMSb}{"6F}

\def\ibind#1{\mathop{#1\mathord{\mathop{\text{--}}}}\!\Ind\nolimits}
\newcommand{\gpdtwst}[2]{\Sigma({#1}, {#2})}

\def\T{\mathbf{T}}
\def\Z{\mathbf{Z}}

\DeclareMathOperator{\Ind}{Ind}

\DeclareMathOperator{\id}{id}
\def\set#1{\{\,#1\,\}}
\newcommand\sset[1]{\{#1\}}
\let\tensor=\otimes

\makeatletter
\def\labelenumi{\textnormal{(\@alph\c@enumi)}}
\def\theenumi{\@alph \c@enumi}
\def\labelenumii{\textnormal{(\@roman\c@enumii)}}
\def\theenumii{\@roman \c@enumii}
\newcount\charno
\def\alphapart#1{\charno=96
\advance\charno by#1\char\charno}

\makeatother

\def\<{\langle}
\def\>{\rangle}
\let\ipscriptstyle=\scriptscriptstyle
\def\lipsqueeze{{\mskip -3.0mu}}
\def\ripsqueeze{{\mskip -3.0mu}}
\def\ipcomma{\nobreak\mathrel{,}\nobreak}
\newbox\ipstrutbox
\setbox\ipstrutbox=\hbox{\vrule height8.5pt
width 0pt}
\def\ipstrut{\copy\ipstrutbox}
\def\lip#1<#2,#3>{\mathopen{\relax_{\ipstrut\ipscriptstyle{
#1}}\lipsqueeze
\langle} #2\ipcomma #3 \rangle}
\def\blip#1<#2,#3>{\mathopen{\relax_{\ipstrut
\ipscriptstyle{ #1}}\lipsqueeze\bigl\langle} #2\ipcomma #3 \bigr\rangle}
\def\rip#1<#2,#3>{\langle #2\ipcomma #3
\rangle_{\ripsqueeze\ipstrut\ipscriptstyle{#1}}}
\def\brip#1<#2,#3>{\bigl\langle #2\ipcomma #3
\bigr\rangle_{\ripsqueeze\ipstrut\ipscriptstyle{#1}}}
\def\angsqueeze{\mskip -6mu}
\def\smangsqueeze{\mskip -3.7mu}
\def\trip#1<#2,#3>{\langle\smangsqueeze\langle #2\ipcomma #3
\rangle\smangsqueeze\rangle_{\ripsqueeze\ipstrut\ipscriptstyle{#1}}}
\def\btrip#1<#2,#3>{\bigl\langle\angsqueeze\bigl\langle #2\ipcomma
#3
\bigr\rangle
\angsqueeze\bigr\rangle_{\ripsqueeze\ipstrut\ipscriptstyle{#1}}}
\def\tlip#1<#2,#3>{\mathopen{\relax_{\ipstrut\ipscriptstyle{
#1}}\lipsqueeze \langle\smangsqueeze\langle} #2\ipcomma #3
\rangle\smangsqueeze\rangle}
\def\btlip#1<#2,#3>{\mathopen{\relax_{\ipstrut\ipscriptstyle{
#1}}\lipsqueeze
\bigl\langle\angsqueeze\bigl\langle} #2\ipcomma #3
\bigr\rangle\angsqueeze\bigr\rangle}

\def\ip(#1|#2){(#1\mid #2)}
\def\bip(#1|#2){\bigl(#1 \mid #2\bigr)}
\def\Bip(#1|#2){\Bigl( #1 \bigm| #2 \Bigr)}
\expandafter\ifx\csname BibSpec\endcsname\relax\else
\BibSpec{collection.article}{%
    +{}  {\PrintAuthors}                {author}
    +{,} { \textit}                     {title}
    +{.} { }                            {part}
    +{:} { \textit}                     {subtitle}
    +{,} { \PrintContributions}         {contribution}
    +{,} { \PrintConference}            {conference}
    +{}  {\PrintBook}                   {book}
    +{,} { }                            {booktitle}
    +{,} { }                            {series}
    +{,} { \voltext}                    {volume}
    +{,} { }                            {publisher}
    +{,} { }                            {organization}
    +{,} { }                            {address}
    +{,} { \PrintDateB}                 {date}
    +{,} { pp.~}                        {pages}
    +{,} { }                            {status}
    +{,} { \PrintDOI}                   {doi}
    +{,} { available at \eprint}        {eprint}
    +{}  { \parenthesize}               {language}
    +{}  { \PrintTranslation}           {translation}
    +{;} { \PrintReprint}               {reprint}
    +{.} { }                            {note}
    +{.} {}                             {transition}
}
\BibSpec{article}{%
    +{}  {\PrintAuthors}                {author}
    +{,} { \textit}                     {title}
    +{.} { }                            {part}
    +{:} { \textit}                     {subtitle}
    +{,} { \PrintContributions}         {contribution}
    +{.} { \PrintPartials}              {partial}
    +{,} { }                            {journal}
    +{}  { \textbf}                     {volume}
    +{}  { \PrintDatePV}                {date}
    +{,} { \eprintpages}                {pages}
    +{,} { }                            {status}
    +{,} { \PrintDOI}                   {doi}
    +{,} { available at \eprint}        {eprint}
    +{}  { \parenthesize}               {language}
    +{}  { \PrintTranslation}           {translation}
    +{;} { \PrintReprint}               {reprint}
    +{.} { }                            {note}
    +{.} {}                             {transition}
}
\BibSpec{book}{%
    +{}  {\PrintPrimary}                {transition}
    +{,} { \textit}                     {title}
    +{.} { }                            {part}
    +{:} { \textit}                     {subtitle}
    +{,} { \PrintEdition}               {edition}
    +{}  { \PrintEditorsB}              {editor}
    +{,} { \PrintTranslatorsC}          {translator}
    +{,} { \PrintContributions}         {contribution}
    +{,} { }                            {series}
    +{,} { \voltext}                    {volume}
    +{,} { }                            {publisher}
    +{,} { }                            {organization}
    +{,} { }                            {address}
    +{,} { pp.~}                        {pages}
    +{,} { \PrintDateB}                 {date}
    +{,} { }                            {status}
    +{}  { \parenthesize}               {language}
    +{}  { \PrintTranslation}           {translation}
    +{;} { \PrintReprint}               {reprint}
    +{.} { }                            {note}
    +{.} {}                             {transition}
}
\fi
\evensidemargin=\oddsidemargin
\mathtoolsset{showonlyrefs} 
\newcommand\U{\mathscr{U}}
\newcommand\V{\mathscr{V}}
\newcommand\W{\mathscr{W}}
\newcommand\sheaffont{\mathscr}
\newcommand\uG{\sheaffont{G}}
\newcommand\uT{\sheaffont{S}}
\newcommand\gu{\Gamma_{\U}}
\newcommand\gv{\Gamma_{\V}}
\newcommand\cnu{\varphi_{\nu}}
\newcommand\cnuc{\varphi_{\nucp}}
\newcommand\nucp{\nu^{c}}

\newcommand\hG{\widehat{G}}

\newcommand\go{\Sigma^{(0)}}
\newcommand\gugo{\Sigma\backslash\go}

\newcommand\A{\mathcal A}

\newcommand\RR{\mathcal{R}}

\newcommand\I{\mathcal{I}}

\newcommand\grh{G}
\newcommand\shh{\mathscr{G}}
\newcommand\aaa{\mathfrak{a}}

\newcommand\indsu{\Ind}
\newcommand\Sp{\Sigma'}
\newcommand\atensor{\odot}
\newcommand\lambdap{\underline{\lambda}}
\newcommand\ssyp{\Sigma(Y')}
\usepackage[normalem]{ulem} 
\usepackage{color}
\definecolor{Dgreen}{cmyk}{0.93,0.33,0.92,0.25} 

\newcommand\dpwon{\bgroup \color{blue}}
\newcommand\dpwoff{\egroup}
\begin{document}
\title[The Dixmier--Douady Class]{The Dixmier--Douady Classes of Certain
  Groupoid \cs-Algebras with Continuous Trace}

\author[M. Ionescu]{Marius Ionescu}

\address{Department of Mathematics\\  United States Naval Academy\\
  Annapolis, MD 21402 USA}

\email{ionescu@usna.edu}

\author[A. Kumjian]{Alex Kumjian}

\address{Department of Mathematics \\ University of Nevada\\ Reno NV
  89557 USA}

\email{alex@unr.edu}

\author[A. Sims]{Aidan Sims}

\address{School of Mathematics and Applied Statistics\\ University of
  Wollongong\\ NSW 2522, Australia}

\email{asims@uow.edu.au}

\author[D. P. Williams]{Dana P. Williams}

\address{Department of Mathematics\\ Dartmouth College \\ Hanover, NH
03755-3551 USA}

\email{dana.williams@Dartmouth.edu}

\date{30 December 2017}

\thanks{This research was supported by the Australian Research
  Council, grant DP150101595. This work was also partially supported
  by Simons Foundation Collaboration grants \#209277 (MI), \#353626
  (AK) and \#507798 (DPW), and by a Junior NARC grant from the United
  States Naval Academy.}

\begin{abstract}
  Given a locally compact abelian group $G$, we give an explicit
  formula for the Dixmier--Douady invariant of the $C^*$-algebra of
  the groupoid extension associated to a \v{C}ech $2$-cocycle in the
  sheaf of germs of continuous $G$-valued functions. We then exploit
  the blow-up construction for groupoids to extend this to some more
  general central extensions of \'etale equivalence relations.
\end{abstract}

\maketitle
\section{Introduction}
\label{sec:introduction} This article provides explicit formulas for
the Dixmier--Douady invariants of a large class of continuous-trace
\cs-algebras arising from groupoid extensions. Continuous-trace
\cs-algebras are amongst the best understood and most intensively
studied classes of Type~I \cs-algebras. A \cs-algebra $A$ is a
continuous-trace \cs-algebra if it has Hausdorff spectrum and is
locally Morita equivalent to a commutative \cs-algebra. Alternatively,
$A$ is of continuous trace if it has Hausdorff spectrum and every
irreducible representation of $A$ admits a neighbourhood $U$ and an
element $a$ of $A$ such that the image of $a$ under any element of $U$
is a rank-one projection. The celebrated Dixmier--Douady theorem
\cite{dixdou:bsmf63} associates to each continuous-trace \cs-algebra
$A$ with spectrum $X$ an element $\delta(A)$ of $H^3(X, \Z)$ such that
$A$ is Morita equivalent to an abelian $C^*$-algebra if and only if
$\delta(A) = 0$. Indeed, the collection of Morita-equivalence classes
of continuous-trace \cs-algebras with spectrum $X$ forms a group under
a balanced tensor-product operation, and $A \mapsto \delta(A)$ induces
a group isomorphism of this group with $H^{3}(X, \Z)$---see
\cite{rw:morita}*{Theorem~6.3}.

As a result, there has been a great deal of work characterizing when
\cs-algebras associated to dynamical systems have continuous
trace. For example, \citelist{\cite{gre:pjm77}
  \cite{wil:jfa81}\cite{ech:tams93} \cite{muhwil:jot84}
  \cite{raewil:cjm93} \cite{raeros:tams88}} investigate when a
transformation group \cs-algebra is continuous-trace; the epic
\cite{ech:mams96} deals with general crossed product \cs-algebras; and
\citelist{\cite{muhwil:jams04} \cite{muhwil:plms395}
  \cite{muhwil:ms92} \cite{muhwil:ms90} \cite{mrw:tams96}} study when
groupoid \cs-algebras have continuous trace.

However, there are few results that provide tools for calculating the
Dixmier--Douady invariant of a given continuous-trace \cs-algebra; and
the results that do address this question are not entirely
satisfactory. For example, in the examples appearing in
\cite{gre:pjm77} and \cite{muhwil:ms90}, the Dixmier--Douady class is
always trivial. And the formula developed in \cite{raewil:cjm93} is
somewhat unwieldy. There are, on the other hand, some intriguing
formulas in \cite{raeros:tams88}, and an explicit computation in
\cite{ionkum16}*{Example~4.6} based on one of these.

There is a simple reason for the dearth of results that compute
$\delta(A)$: the Dixmier--Douady invariant is difficult to compute
except on an \emph{ad hoc} basis. In this note, we take some steps
towards addressing this lack of computable examples, by developing
usable formulas for the Dixmier--Douady class of significant classes
of groupoid \cs-algebras. Starting with a locally compact abelian
group $G$ and a space $X$, we first consider groupoids constructed
directly from a \v{C}ech cocycle on $X$ taking values in the sheaf of
germs of continuous $G$-valued functions on $X$. We then extend this
to central groupoid extensions of equivalence relations constructed
from local homeomorphisms between locally compact Hausdorff spaces.
Not surprisingly, our most elegant results are obtained under more
restrictive hypotheses; but even in the more general situation our
computation is fairly concrete. In any case, the subject of operator
algebras, and the study of \cs-algebras associated to dynamical
systems in particular, is short on concrete examples, so we think that
extra hypotheses are worth it.

The main results of the paper are as follows. Consider a
second-countable locally compact Hausdorff space $X$ and a
second-countable locally compact abelian group $G$. Take a \v{C}ech
$2$-cocycle $c$ on $X$, relative to a locally finite open cover
$\U = \{U_i : i \in I\}$ of $X$, taking values in the sheaf $\shh$ of
germs of continuous $G$-valued functions on $X$. The associated
Raeburn--Taylor groupoid $\gu$ consists of triples $(i, x, j)$ such
that $x \in U_{ij} \subseteq X$. The cocycle $c$ determines a natural
central extension $\Sigma_{c}$ of $\gu$ by $G$: as a set $\Sigma_{c}$
is just a copy of $\gu \times G$, but the composition in the second
coordinate is twisted by the cocycle $c$ as in
equation~\eqref{eq:Sigma-c mult}. Our first main result,
Theorem~\ref{thm-main-dd-calc}, says that $\cs(\Sigma_c)$ is a
continuous-trace algebra with spectrum $\hG\times X$, and computes its
Dixmier--Douady invariant as follows: write $\uT$ for the sheaf of
germs of continuous $\T$-valued functions on $\hG \times X$, and let
$\V$ be the cover $\set{\hG \times U_i}_{i \in I}$ of $\hG \times
X$. Then the cocycle $c$ determines a cocycle
$\nucp \in Z^2(\hG \times X, \uT)$ relative to $\V$ such that
$\nucp_{ijk}(\tau, x) = \overline{\tau\bigl(c_{ijk}(x)\bigr)}$ for all
$i,j,k \in I$, $x \in U_{ijk}$ and $\tau \in \hG$.  The assignment
$c \mapsto \nucp$ descends to a homomorphism
\begin{equation}
  m_{*}: H^{2}(X,\uG)\to H^{2}(\hG \times X, \uT),
\end{equation}
and Theorem~\ref{thm-main-dd-calc} shows that, under the canonical
isomorphism of $H^3(\hG \times X, \Z)$ with $H^2(\hG \times X, \uT)$,
the Dixmier--Douady class of $\cs(\Sigma_c)$ is carried to the
cohomology class $[\nucp]$. The set-up and proof of this theorem
occupy Sections
\ref{sec:nice-class-examples}--\ref{sec:dixmier-douady-class}.

We then build upon Theorem~\ref{thm-main-dd-calc} in
Section~\ref{sec:local-lomeomorphisms} to describe a method for
computing the Dixmier--Douady invariants of more general central
extensions. We start with a local homeomorphism $\psi : Y \to X$ of
second-countable locally compact Hausdorff spaces, and form the
equivalence relation $R(\psi)$ on $Y$ consisting of pairs with
identical image under $\psi$. We consider a central extension $\Sigma$
of $R(\psi)$ by $G \times Y$. We assume that $X$ is locally
$G$-trivial in the sense that every open cover of $X$ admits a
refinement such that, on double overlaps, every principle $G$-bundle
is trivial. This hypothesis ensures that a suitable blowup $\Sigma'$
of the extension $\Sigma$ admits a continuous section for the
surjection onto the corresponding blowup $R'$ of $R(\psi)$. It follows
that $\Sigma'$ is determined by a continuous $G$-valued groupoid
$2$-cocycle on $R'$. A little more work puts us back in the situation
of Theorem~\ref{thm-main-dd-calc}, and we can use this to compute the
Dixmier--Douady invariant of $\cs(\Sigma')$. The blowup operation
determines an equivalence of extensions, and hence a Morita
equivalence of their \cs-algebras, yielding a computation of
$\delta(\cs(\Sigma))$.

For the reader's convenience we include an appendix with background on
central extensions of groupoids by locally compact abelian groups and,
in particular, those that arise from continuous 2-cocycles (see
Appendix~\ref{sec:Extcoc}).

The second and fourth authors thank the third author for his
hospitality and support during trips to the University of Wollongong.

\section{Central Isotropy}
\label{sec:central-isotropy}

In the sequel, $\Sigma$ will always be a second-countable locally
compact Hausdorff groupoid with a Haar system
$\sset{\lambda^{u}}_{u\in\go}$.  The \emph{isotropy groupoid} of
$\Sigma$ is the closed subgroupoid
\begin{equation}
  \label{eq:1n}
  \I(\Sigma) = \set{\gamma\in \Sigma:s(\gamma)=r(\gamma)}.
\end{equation}
Note that $\I(\Sigma)$ is a group bundle over $\go$ and that
$\I(\Sigma)$ admits a Haar system if and only if the isotropy map
\begin{equation}
  \label{eq:2}
  u\mapsto \Sigma(u)=\set{\gamma\in \I(\Sigma) : s(\gamma)=u=r(\gamma)}
\end{equation}
is continuous from $\go$ into the locally compact Hausdorff space
$\mathcal{C}(\Sigma)$ of closed subgroups of $\Sigma$
\cite{ren:jot91}*{Lemma~1.3}. In the sequel we need to assume not only
that each $\Sigma(u)$ is abelian, but that the isotropy is central in
the following sense.

\begin{definition}
  Let $\Sigma$ be a groupoid, $\I(\Sigma)$ its isotropy subgroupoid,
  and $q:\go\to\gugo$ the quotient map.  We say that $\Sigma$ has
  \emph{central isotropy} if $\gugo$ is Hausdorff and there is an
  abelian group bundle $\A$ over $\gugo$ and a groupoid isomorphism
  $\iota:q^{*} \A \to \I(\Sigma)$ such that $\iota|_{\go}=\id$ and
  such that
  \begin{equation}
    \iota\bigl(r(\gamma),a\bigr) \gamma = \gamma\iota\bigl(s(\gamma),a
    \bigr)
    \quad\text{for all $a\in A([r(\gamma)])$.}
  \end{equation}
\end{definition}

The use of the word ``central'' is partially justified by the
following example.

\begin{example}\label{ex-central-haar}
  Let $(G,X)$ be a transformation group with $G\backslash X$ Hausdorff
  and such that each stability group $G_{x}=\set{h\in G:h\cdot x =x}$
  is central in $G$.  Let $\Sigma=G\times X$ be the corresponding
  transformation groupoid.  Note that $G_{h\cdot x}=G_{x}$ for all
  $h\in G$ and $x\in X$.  Then
  \begin{equation}
    \A=\set{(G\cdot x,h)\in G\backslash X\times G : h\in G_{x}}
  \end{equation}
  is a group bundle over $G\backslash X$.  If $q:X\to G \backslash X$
  is the orbit map, then $\Sigma$ has 
  central isotropy with respect to the isomorphism
  $\iota(x,G\cdot x, h)=(h,x)$.
\end{example}

An important class of examples comes from \emph{$\T$-groupoids} or
\emph{twists} as introduced by the second author
\citelist{\cite{kum:lnim85} \cite{kum:cjm86}}.  Recall that a
$\T$-groupoid $\Sigma$ over a groupoid $\RR$ is a
unit-space-preserving groupoid extension
\begin{equation}\label{eq:t-gr}
  \begin{tikzcd}[column sep=3cm]
    \go\times\T \arrow[r,"\iota", hook] \arrow[dr,shift left, bend
    right = 15] \arrow[dr,shift right, bend right = 15]&\Sigma
    \arrow[r,"\pi", two heads] \arrow[d,shift left] \arrow[d,shift
    right]&\RR \arrow[dl,shift left, bend left = 15] \arrow[dl,shift
    right, bend left = 15]
    \\
    &\go&
  \end{tikzcd}
\end{equation}
such that
\begin{equation}
  \label{eq:3n}
  \iota\bigl(r(\gamma),z\bigr)\gamma
  =\gamma\iota\bigl(s(\gamma),z\bigr)\quad\text{for all $\gamma\in \Sigma$
    and $z\in\T$.}
\end{equation}

Given a $\T$-groupoid $\Sigma$ over $\RR$, there is a free and proper
$\T$-action on $\Sigma$ such that
$z\cdot \gamma = \iota\bigl(s(\gamma),z\bigr)\gamma$. We can identify
$\RR$ with the orbit space $\T\backslash \Sigma$ and $\pi$ with the
orbit map---see \cite{muhwil:plms395}*{\S3}. See Appendix
\ref{sec:Extcoc} for more background on extensions of groupoids by
locally compact abelian groups.

\begin{example}\label{ex-t-gp}
  Recall that a groupoid $\RR$ is \emph{principal} if the isotropy
  groupoid $\I(\RR)$ is just $\RR^{(0)}$.  Let $\Sigma$ be a
  $\T$-groupoid over a principal groupoid $\RR$.  Then
  $\iota(\go\times\T) = \I(\Sigma)$, and if $\gugo$ is Hausdorff, then
  $\Sigma$ has central isotropy.
\end{example}

\begin{remark}
  In previous studies of $\T$-groupoids, the emphasis was on the
  quotient \cs-algebra $\cs(\RR;\Sigma)$.  Here we focus on
  $\cs(\Sigma)$.
\end{remark}

A central player in much of the work on the Effros--Hahn theory for
groupoids---as in \cite{ren:jot91} and \cite{ionwil:iumj09}---is the
\emph{equivalence relation $\RR $ associated to $\Sigma$}. By
definition, this $\RR$ is the image of the map
$\pi:\Sigma\to \go\times\go$ given by
$\pi(\gamma)=\bigl(r(\gamma),s(\gamma)\bigr)$. Since the relative
product topology on $\RR$ is unlikely to be useful in general, it is
common to equip $\RR$ with the quotient topology, which is finer
(often strictly finer) than the relative product topology. Even then,
$\RR$ need not be a tractable topological space.  However, the
isotropy groupoid $\I(\Sigma)$ acts on the right and left of $\Sigma$,
and with respect to the quotient topologies
\begin{equation}
  \RR \cong \I(\Sigma)\backslash \Sigma = \Sigma/\I(\Sigma).
\end{equation}
As observed above, $\I(\Sigma)$ has a Haar system precisely when
$u\mapsto \Sigma(u)$ is continuous.  In this case, the orbit map
$k:\Sigma\to \I(\Sigma)\backslash \Sigma$ is open (see
\cite{muhwil:plms395}*{Lemma~2.1}),\footnote{The subtlety here is that
  the orbit map for an action by a groupoid $\Gamma$ is open
  \emph{provided} the range map of $\Gamma$ is open.  This is
  automatic if $\Gamma$ has a Haar system.} and then $\RR $ is locally
compact and Hausdorff.

If $\RR $ is locally compact, then we can ask for $\RR $ to act
properly on its unit space, which we identify with $\go$.  In that
case, \cite{muhwil:ms90}*{Lemma~2.1} shows that we can identify $\RR$
with $\pi(\Sigma)$ under the relative topology inherited from
$\go\times \go$.

Hence we work under three key assumptions.
\begin{compactenum}
\item [(A1)] $\Sigma$ has central isotropy.
\item [(A2)] The isotropy map $u\mapsto \Sigma(u)$ is continuous.
\item [(A3)] The action of $\RR $ on $\go$ is proper.
\end{compactenum}
If $\Sigma$ has central isotropy and $\cs(\Sigma)$ has continuous
trace, then (A2)~and~(A3) are automatically satisfied
\cite{mrw:tams96}*{Theorem~1.1}.

In any event, we have a unit-preserving short exact sequence of
groupoids
\begin{equation}
  \begin{tikzcd}[column sep=3cm]
    q^{*}\A \arrow[r,"\iota", hook] \arrow[dr,shift left, bend right =
    15] \arrow[dr,shift right, bend right = 15]&\Sigma \arrow[r,"\pi",
    two heads] \arrow[d,shift left] \arrow[d,shift right]&\RR
    \arrow[dl,shift left, bend left = 15] \arrow[dl,shift right, bend
    left = 15]
    \\
    &\go&
  \end{tikzcd}
\end{equation}
similar to \eqref{eq:t-gr}. We study special cases of these sorts of
groupoids in the next four sections.

\begin{remark}[Irreducible Representations Induced from Characters]
  \label{rem-rho-reps}
  In our main results, we will make considerable use of irreducible
  representations of $\cs(\Sigma)$ induced from characters on a
  stability group as described in \cite{ionwil:pams08}*{\S2}.  If
  $\tau$ is a character on the stability group $\Sigma(u)$, then we will
  write $\indsu^{\Sigma}(u,\tau)$, or simply $\indsu(u,\tau)$ when
  there is no ambiguity about $\Sigma$, for the induced represenation
  $\Ind_{\Sigma(u)}^{\Sigma}(\tau)$.  Then $\indsu(u,\tau)$ is
  irreducible by \cite{ionwil:pams08}*{Theorem~5}. If
  $\sset{\lambda^{u}}_{u\in\go}$, is a Haar system on $\Sigma$ and
  $\mu$ is a Haar measure on $\Sigma(u)$, then $\indsu(u,\tau)$ acts
  by convolution on the completion of $C_{c}(\Sigma_{u})$ with respect
  to to the pre-inner product
  \begin{equation}
    \label{eq:suip}
    \ip(f_{1}|f_{2})= \int_{\Sigma(u)} \int_{\Sigma}
    \overline{f_{2}(\sigma^{-1})}
    f_{1}(\sigma^{-1}g)\,d\lambda^{u}(\sigma) \tau(g)\,d\mu(g).
  \end{equation}
  If, as will always be the case here, the orbits $[u]=\Sigma\cdot u$
  are closed, then every irreducible representation of $\cs(\Sigma)$
  factors through a restriction $\Sigma([u])=\Sigma^{[u]}_{[u]}$.
  Since the latter is equivalent as a groupoid to the isotropy group
  $\Sigma(u)$, it is easy to see directly that $\indsu(u,\tau)$ is
  irreducible.  If all the isotropy groups are abelian, then it is
  also clear that every irreducible representation on $\cs(\Sigma)$ is
  of this form for some $u\in\go$ and $\tau\in\Sigma(u)^{\wedge}$.
  Furthermore, if $[u]=[v]$ and $\sigma\in \Sigma_{u}^{v}$ then
  $\indsu(u,\tau)$ is easily seen to be equivalent to
  $\indsu(v,\sigma\cdot \tau)$ where 
  $\sigma\cdot\tau(g)= \tau(\sigma^{-1} g \sigma)$.  If we have
  central isotropy, so that we can identify $\Sigma(u)$ and
  $\Sigma(\sigma\cdot u)$, then the spectrum of $\cs(\Sigma)$ is
  parameterized by $\set{([u],\tau): \text{$u\in\go$ and
      $\tau\in \Sigma(u)^{\wedge}$}}$.
\end{remark}

\section{A Class of Examples}
\label{sec:nice-class-examples}

Let $X$ be a second-countable locally compact Hausdorff space and
$\grh$ a second-countable locally compact abelian group. Let $\shh$ be
the sheaf of germs of $\grh$-valued functions on $X$ (see
\cite{rw:morita}*{\S4.1}). Let $\aaa$ be an element in the sheaf
cohomology group $H^{2}(X,\shh)$.  Then $\aaa $ is represented by a
two cocycle $c\in Z^{2}(\U,\shh)$ for some \emph{locally finite} cover
$\U=\set{U_{i}}_{i\in I}$ of $X$ by  precompact open
sets.  We say that $c$ is \emph{normalized} if $c_{iii}(x)=1$
for all~$i$ and all $x\in U_{i}$.  It is easy to see that every
$2$-cocycle is cohomologous to a normalized one and we will assume all
our cocycles are normalized.  We record some elementary facts about
normalized cocycles for reference.
\begin{lemma}
  \label{lem-norm-coc} Let $c\in Z^{2}(\U,\shh)$ be normalized.  Then
  for all $i,j,k\in I$,
  \begin{enumerate}
  \item $c_{iij}(x)=c_{ijj}(x)=0$.
  \item $c_{iji}(x)=c_{jij}(x)$.
  \item $c_{ijk}(x)=-c_{jik}(x) +c_{iji}(x)$.
  \item $c_{ijk}(x)=-c_{ikj}(x)+c_{jkj}(x)$.
  \item $c_{iji}(x)+c_{jki}(x)= -c_{ikj}(x)+c_{iki}(x)+c_{jkj}(x)$.
  \end{enumerate}
\end{lemma}
\begin{proof}
  These are all straightforward consequences of the cocycle identity.
  For example, (a) follows from computations like
  $0=\delta(c)_{iiji}(x)=
  c_{iji}(x)-c_{iji}(x)+c_{iii}(x)-c_{iij}(x)$.  For (b), consider
  $\delta(c)_{ijij}(x)$ and use~(a).  For (c), consider
  $\delta(c)_{ijik}(x)$ and (d) follows from $\delta(c)_{ijkj}(x)$.
  We get~(e) using (c)~and (d):
  \begin{align}
    c_{ikj}(x)&=-c_{ijk}(x)+c_{jkj}(x)= c_{jik}(x)-c_{iji}(x)+c_{jkj}(x)
    \\
              &=
                -c_{jki}(x)+c_{iki} (x)-c_{iji}(x) +c_{jkj}(x).  \qedhere
  \end{align}
\end{proof}

Given $\U$, we can form the blow-up groupoid\footnote{The terms
  ``pull-back'' and ``ampliation'' are also used. Blow-ups are defined
  and discussed in \cite{wil:pams16}*{\S3.3}.}  $\gu$ with respect to
the natural map of $\coprod U_{i}$ onto $X$:
\begin{align}
  \SwapAboveDisplaySkip
  \label{eq:5n}
  \gu=\set{(i,x,j):x\in U_{ij}:= U_{i}\cap U_{j}}
\end{align}
with $(i,x,j)(j,x,k)=(i,x,k)$ and $(i,x,j)^{-1}=(j,x,i)$.  In
particular, $\gu$ is a principal groupoid with unit space
$\gu^{(0)} = \coprod U_{i}$, and is equivalent to the space $X$ (see
\citelist{\cite{raetay:jamss85} \cite{hae:a84}}).

Let $\Sigma_{c}$ be the groupoid extension equal as a topological
space to $G\times\gu$ endowed with the operations
\begin{equation}\label{eq:Sigma-c mult}
  \bigl(g,(i,x,j)\bigr)\bigl(h,(j,x,k)\bigr) = \bigl(g+h+c_{ijk}(x),
  (i,x,k) \bigr)
\end{equation}
and
\begin{equation}\label{eq:7}
  \bigl(g,(i,x,j)\bigr)^{-1}= \bigl( -g-c_{iji}(x), (j,x,i)\bigr).
\end{equation}
Since
\begin{align}
  \bigl(g,(i,x,j)\bigr)\bigl(g,(i,x,j)\bigr)^{-1}
  =
  \bigl(g,(i,x,j)\bigr) \bigl(-g-c_{iji}(x),(j,x,i)\bigr) 
  = \bigl(0
  ,(i,x,i)\bigr),
\end{align}
and similarly
\begin{equation}
  \bigl(g,(i,x,j)\bigr)^{-1}\bigl(g,(i,x,j)\bigr)=\bigl(0,(j,x,j)\bigr),
\end{equation}
we can identify the unit space of $\Sigma_{c}$ with $\coprod
U_{i}$. Let $\mu$ be a Haar measure on $G$.  We equip $\Sigma_{c}$
with the Haar system $\lambda=\sset{\lambda^{(i,x)}}$ given by
\begin{equation}
  \label{eq:37}
  \lambda^{(i,x)}(f)=\sum_{j}\int_{G} f\bigl(g,(i,x,j)\bigr)\,d\mu(g).
\end{equation}
Then using Lemma~\ref{lem-norm-coc},
\begin{align}
  f*f'\bigl(g,&(i,x,j)\bigr)\\
              &=\sum_{k}\int_{G} f\bigl(h,(i,x,k)\bigr)
                f'\bigl( g-h-c_{iki}(x)+c_{kij}(x),(k,x,j)\bigr)\,d\mu(h) \\
              &= \sum_{k}\int_{G} f\bigl(h,(i,x,k)\bigr)
                f'\bigl( g-h-c_{ikj}(x),(k,x,j)\bigr)\,d\mu(h)\label{eq:11}
\end{align}
while
\begin{equation}
  f^{*}\bigl(g,(i,x,j)\bigr) = \overline{f\bigl(-g-c_{iji}(x),(j,x,i)\bigr)}.
\end{equation}

If we define $\iota :G\times\coprod U_{i}\to \Sigma_{c}$ by
\begin{equation}
  \iota\bigl(g,(i,x)\bigr) = \bigl(g,(i,x,i)\bigr),
\end{equation}
and $j:\Sigma_{c}\to \gu$ by
\begin{equation}
  j\bigl(g,(i,x,j)\bigr)= (i,x,j),
\end{equation}
then we obtain a groupoid extension
\begin{equation}\label{eq:36}
  \begin{tikzcd}[column sep=3cm]
    G\times\coprod U_{i} \arrow[r,"\iota", hook] \arrow[dr,shift left,
    bend right = 15] \arrow[dr,shift right, bend right =
    15]&\Sigma_{c} \arrow[r,"\pi", two heads] \arrow[d,shift left]
    \arrow[d,shift right]&{\gu.}\arrow[dl,shift left, bend left = 15]
    \arrow[dl,shift right, bend left = 15]
    \\
    &\coprod U_{i}&
  \end{tikzcd}
\end{equation}
We think of $\Sigma_{c}$ as a generalized twist in which $\T$ has been
replaced by $G$.  As in Remark~\ref{rem-rho-reps}, we can identify the
spectrum of $\cs(\Sigma_{c})$ as a set with $\hG\times X$ via
$(\tau,x) \mapsto [\indsu((i,x),\tau)]$ for any $i$ such that
$x\in U_{i}$.

\begin{lemma}
  \label{lem-rho-reps} Let $\Sigma_{c}$ be as above.  Let
  $I(x)=\set{j\in I:x\in U_{j}}$. Then $\indsu((i,x),\tau)$ is
  equivalent to the representation $L$ on $\ell^{2}(I(x))$ where
  $L(f)$ is given by multiplication by the matrix $A=(a_{jk})$ with
  \begin{equation}
    a_{jk}=\tau{(c_{ijk}(x))} \int_{G}f(g,(j,x,k))\tau(g)\,d\mu(g).
  \end{equation}
\end{lemma}
\begin{proof}
  We have $(\Sigma_{c})_{(i,x)}=\set{(g,(j,x,i):j\in I(x)}$.  Then
  $\indsu((i,x),\tau)$ acts by convolution on the completion of
  $C_{c}\bigl((\Sigma_{c})_{(i,x)}\bigr)$ with respect to the inner
  product given in \eqref{eq:suip}:
  \begin{align}
    \SwapAboveDisplaySkip
    \bigl(&f_{1}\mid f_{2}\bigr)\\
          &= \sum_{j}\int_{G}\int_{G} \overline{f_{2}(-h-c_{iji}(x),
            (j,x,i))} f_{1}(h^{-1}-c_{iji}(x) + g,(j,x,i)) \tau(g)
              \,d\mu(h)\,d\mu(g) \\
          &= \sum_{j}U(f_{1})(j) \overline{U(f_{2})(j)},
  \end{align}
  where
  \begin{equation}
    U(f)(j)=\int_{G} f(g-c_{iji}(x),(j,x,i))\tau(g)\,d\mu(g).
  \end{equation}

  Hence $U$ is a unitary from the space of $\indsu((i,x),\tau)$ onto
  $\ell^{2}(I(x))$.  But
  \begin{align}\SwapAboveDisplaySkip
    \label{eq:42}
    U(&f_{1}*f_{2})(j) \\
      &=
        \int_{G} f_{1}*f_{2}(g-c_{iji}(x), (j,x,i))\tau(g)\,d\mu(g) \\
      &= \sum_{k} \int_{G}\int_{G} 
        f_{1}(h,(j,x,k)) f_{2}(-h+g-c_{iji}(x) -c_{jki}(x),(k,x,i))
        \tau(g)
            \,d\mu(h)\,d\mu(g) \\
    \intertext{which, since $c_{iji}(x)+c_{jki}(x) =
    c_{ijk}(x)+c_{iki}(x)$, is}
      &=\sum_{k} \int_{G}\int_{G}
        f_{1}(h,(j,x,k)) f_{2}(-h+g-c_{ijk}(x) -c_{iki}(x),(k,x,i))
        \tau(g)
            \,d\mu(h)\,d\mu(g) \\
      &=\sum_{k}\tau(c_{ijk}(x))\int_{G}f_{1}(h,(j,x,k))\tau(h)\,d\mu(h)
        U(f_{2})(k)\\
      &=\sum_{k}a_{jk}U(f_{2})(k).
  \end{align}
  Thus $U$ intertwines $\indsu((i,x),\tau)$ with multiplication by
  $A=(a_{jk})$ as claimed.
\end{proof}

\begin{remark}
  There is continuous groupoid 2-cocycle $\varphi_c \in Z^2(\gu, G)$
  given by the formula
  $\varphi_c\bigl((i,x,j),(j,x,k)\bigr) = c_{ijk}(x)$ for
  $x \in U_{ijk}$. For this cocycle, $\Sigma_c$ is equal to the
  extension $\gpdtwst{\gu}{\varphi_c}$ described in
  Notation~\ref{not:Esigmadef} under the natural
  identification.
\end{remark}

Our first goal is to determine the Dixmier--Douady class of
$\cs(\Sigma_{c})$. To do so, we need the following construction. Given
a \v{C}ech $2$-cocycle $c \in Z^2(U, \uG)$, where
$\U = \{U_i\}_{i \in I}$ is a locally finite open cover of $X$, the
cover $\V = \{\hG \times U_i\}_{i \in I}$ of $\hG \times X$ is locally
finite and supports a normalized $2$-cocycle $\nucp$ such that
\begin{equation}\label{eq:10}
  \nucp_{ijk}(\tau,x)=\overline{\tau\bigl(c_{ijk}(x)\bigr)}.
\end{equation}

There is a well-defined homomorphism
\begin{equation}\label{eq:9n}
  m_{*}:H^{2}(X,\uG)\to H^{2}(\hG \times X, \uT)
\end{equation}
such that $m_*([c]) = [\nucp]$ for all $c \in Z^2(X, \uG)$.

Our first main theorem is the following computation of the
Dixmier--Douady class of $\cs(\Sigma_c)$.

\begin{thm}
  \label{thm-main-dd-calc} Suppose that $X$ is a second-countable
  locally compact Hausdorff space and that $G$ is a second-countable
  locally compact abelian group. Let $\uG$ be the sheaf of germs of
  continuous $G$-valued functions on $X$. Suppose that  $c\in
  Z^{2}(\U,\uG)$ is a 
  normalized cocycle on a locally finite cover $\U$ by precompact
  open sets representing a class $\aaa\in H^{2}(X,\uG)$, and
  let $\Sigma_{c}$ be the associated groupoid extension \eqref{eq:36}.
  Then $(\tau, x) \mapsto [\indsu((i,x),\tau)]$, $x \in U_i$, is a
  homeomorphism of $\hG\times X$ with the spectrum
  $\cs(\Sigma_{c})^{\wedge}$.  Furthermore $\cs(\Sigma_{c})$ has
  continuous trace, and with respect to this identification of the
  spectrum with $\hG\times X$, its Dixmier-Douady class
  $\delta\bigl(\cs(\Sigma_{c})\bigr)$ is equal to the image of
  $m_{*}(\aaa)$ in $H^{3}(\hG\times X,\Z)$.
\end{thm}

The proof of Theorem~\ref{thm-main-dd-calc} requires some preparation,
including the Raeburn--Taylor construction described in the next
section. So we defer the proof to
Section~\ref{sec:dixmier-douady-class}.

\section{The Raeburn--Taylor Algebra}
\label{sec:raeb-tayl-algebra}

For our computation of the Dixmier-Douady class of $\cs(\Sigma_{c})$,
we want a slight modification of the \emph{Raeburn--Taylor Algebra}
based on their original construction in \cite{raetay:jamss85} and
reproduced in \cite{rw:morita}*{Proposition~5.40}.  Specifically,
given a \emph{normalized} $2$-cocycle
$\nu=\sset{\nu_{ijk}} \in Z^{2}(\U,\uT)$ defined on a \emph{locally
  finite} cover $\U$  by precompact open sets, we want to
produce a concrete \cs-algebra $A(\nu)$ with finite-dimensional
representations such that $\delta(A(\nu))$ is the image of $[\nu]$ in
$H^{3}(X,\Z)$.  In \cite{rw:morita}*{Proposition~5.40} and in
\cite{raetay:jamss85}, it is assumed for convenience that $\nu$ is
alternating.\footnote{It is not stated in \cite{raetay:jamss85} that
  their cocycle is alternating, but it is needed for their
  constructions.  Furthermore, there is a subtle caveat in the
  Raeburn--Taylor construction.  They use the Shrinking Lemma (see
  \cite{rw:morita}*{Lemma~4.32}) to replace $\U$ by a cover
  $\V=\sset{V_{i}}_{i\in I}$ \emph{with the same index set} such that
  $\overline{V_{i}}\subset U_{i}$.  Replacing $\U$ by $\V$ allows them
  to assume that each $\nu_{ijk}$ extends to $\overline{U_{ijk}}$.
  The exposition here---taken from \cite{rw:morita}---avoids this
  technicality.}  Although every cocycle is cohomologous to an
alternating one by \cite{rw:morita}*{Proposition~4.41}, it will
simplify our arguments here to observe that basically the same
constructs work in the normalized case.  We supply some of the details
for completeness and also since we will have to push the envelope a bit
further in Section~\ref{sec:dixmier-douady-class}.

To construct $A(\nu)$, we begin by forming the algebra $ A_{1}(\nu)$
which is the set of sparse $I\times I$-matrices
\begin{equation}
  f=\bigl(f_{ij}\bigr)_{i,j\in I}
\end{equation}
where each $f_{ij}\in C_{0}(X)$ and vanishes off
$U_{ij}$.\footnote{Requiring that each $f_{ij}$ belongs to $C_{0}(X)$
  rather than just to $C(X)$ is redundant if each $U_{ij}$ is
  precompact as we are assuming here.  The key point is that the
  $f_{ij}|_{U_{ij}}$ extend continuously to all of $X$.  We include
  this nicety here so that, later on, we can more easily extend this
  construction to a particular cover by sets that are not necessarily
  precompact.} For each $x\in X$, the set
$I(x):=\set{i\in I:x\in U_{i}}$ is finite.  Thus if $n_{x}= |I(x)|$,
then $\bigl(f_{ij}(x)\bigr)_{i,j\in I}$ is an $n_{x}\times n_{x}$
matrix.  We define a multiplication on $A_{1}(\nu)$ by twisting the
usual matrix multiplication with $\nu$:
\begin{equation}
  (f_{ij} )(g_{lk}) = (h_{ij})
\end{equation}
where
\begin{equation}\label{eq:4}
  h_{ik}(x)=\sum_{j} \overline{\nu_{ijk}(x)}
  f_{ij}(x)g_{jk}(x).
\end{equation}
To see that the sum in \eqref{eq:4} is meaningful, observe that it is
always a finite sum:
\begin{equation}
  h_{ik}(x) =\begin{cases}
    \sum_{\set{j\in I:x\in U_{ijk}}} \overline{\nu_{ijk}(x)}
    f_{ij}(x)g_{jk}(x) &\text{if $x\in U_{ijk}$ for some $j$,
      and}\\
    0&\text{otherwise.}
  \end{cases}
\end{equation}
As in \cite{rw:morita}*{Lemma~5.39}, using the local finiteness of the
cover and the compactness of the $\overline{U_{i}}$, it is not hard to
see that each $h_{ik}$ is continuous and vanishes off $U_{ik}$.  (This
will also be a special case of Lemma~\ref{lem-prod-ok}.)

To get an involution, we have adjust the definition of the involution
on $A(\nu)$ given by Raeburn and Taylor to account for the fact that
our cocycle might not be alternating: we define
\begin{align}\SwapAboveDisplaySkip\label{eq:3}
  f^{*}=(f_{ij})^{*}=(g_{ij})\quad\text{where
  $g_{ij}(x)=\nu_{iji}(x)\overline{f_{ji}(x)}$.}
\end{align}
This map is involutive in view of Lemma~\ref{lem-norm-coc}(b).  To see
that it is anti-multiplicative, we require
Lemma~\ref{lem-norm-coc}(e):
\begin{align}
  \label{eq:17}
  (f*g)^{*}_{ij}(x)
  &= \nu_{iji}(x) \overline{(f*g)_{ji}(x)}\\
  &= \nu_{iji}(x) \sum_{k}\nu_{jki}(x) \overline{f_{jk}(x)}
    \overline{g_{ki}(x)} \\
  \intertext{which, using Lemma~\ref{lem-norm-coc}(e), is}
  &= \sum_{k}\overline{\nu_{ikj}(x)} \nu_{iki}(x) \overline{g_{ki}(x)}
    \nu_{jkj} (x) \overline{f_{jk}(x)} \\
  &= \sum_{k} \overline{\nu_{ikj}(x)} g^{*}_{ik}(x) f^{*}_{kj}(x) \\
  &=(g^{*}*f^{*})_{ij}(x).
\end{align}

We now follow the discussion preceding
\cite{rw:morita}*{Proposition~5.40} \emph{mutatis mutandis}.

%
For each pair $(i,x)\in \coprod U_{i}$, define a representation
$\pi_{(i,x)}$ of $A_{1}(\nu)$ on $\ell^{2}(I(x))$ by letting
$\pi_{(i,x)}(f)$ act by multiplication by the matrix
\begin{equation}
  \bigl(\overline{\nu_{ikl}(x)} f_{kl}(x)\bigr)_{kl}.
\end{equation}
Using some straightforward cocycle identities, we see that
$\pi_{(i,x)}$ is multiplicative and $*$-preserving.  The
representations $\pi_{(i,x)}$ and $\pi_{(j,x)}$ are equivalent, so we
get a seminorm
\begin{equation}
  \|f\|_{x}=\|\pi_{(i,x)}(f)\|\quad\text{for any $i$ such that $x\in
    U_{i}$}.
\end{equation}

Just as in \cite{rw:morita}*{Proposition~5.40} and
\cite{raetay:jamss85}*{Theorem~1}, we can let
\begin{equation}
  A(\nu)=\set{f\in A_{1}(\nu):\text{$x\mapsto \|f\|_{x}$ vanishes at
      infinity on $X$}}.
\end{equation}
Then $\|f\|=\sup_{x}\|f\|_{x}$ is a complete norm on
$A(\nu)$. Furthermore, with respect to this norm, $A(\nu)$ is a
continuous-trace \cs-algebra with spectrum
$X=\set{[\pi_{(i,x)}]: x\in X}$ and Dixmier--Douady class
$[\nu]$. (Here and elsewhere we will often write $[\nu]$ for its image
in $H^{3}(X, \Z)$.)

\begin{remark}[The Raeburn--Taylor Groupoid] \label{rmk:rtgroupoid}
  Nowadays we favor realizing $A(\nu)$ as a groupoid \cs-algebra
  twisted by a continuous $2$-cocycle $\cnu$.  The groupoid is the
  blow-up $\gu$ associated to the cover $\U$ of $X$ corresponding to
  $\nu$
  and the cocycle $\cnu$ in $Z^{2}(\gu,\T)$ is given by
  \begin{equation}\label{eq:5}
    \cnu\bigl((i,x,j),(j,x,k)\bigr)=\overline{\nu_{ijk}(x)}.
  \end{equation}
  (The complex conjugate in \eqref{eq:5} is missing from the formula
  in \cite{raetay:jamss85}.)  Note that $\cnu$ is normalized as in
  Appendix~\ref{sec:Extcoc} since $\nu$ is.
  The operations in $C_{c}(\gu,\cnu)$ are given by
  \begin{align}
    f*g(i,x,k)&=\sum_{j} f(i,x,j)g(j,x,k)\cnu\bigl((i,x,j),(j,x,k)\bigr) \\
              &= \sum_{j} f(i,x,j)g(j,x,k)\overline{\nu_{ijk}(x)},
  \end{align}
  and
  \begin{equation}
    f^{*}(i,x,j)=\overline{f(j,x,i)}
    \overline{\cnu\bigl((i,x,j),(j,x,i)\bigr)} =\nu_{iji}(x)
    \overline{f(j,x,i)} . 
  \end{equation}

  Looking over these formulas, it is immediate that we get a
  $*$-homomorphism $\Phi:C_{c}(\gu,\cnu)\to A(\nu)$ given by
  \begin{equation}\label{eq:6}
    \Phi(f)_{ij}(x)=f(i,x,j).
  \end{equation}
  Just as observed in \cite{raetay:jamss85}*{Remark~3}, this map extends to an
  isomorphism.
  \begin{prop}
    [\cite{raetay:jamss85}*{Remark~3}] The map $\Phi$ extends to an
    isomorphism of $\cs(\gu,\cnu)$ onto $A(\nu)$.
  \end{prop}

  The proof is essentially the same, but easier, than the proof of
  Theorem~\ref{thm-main-dd-calc} below.
\end{remark}

\section{The Dixmier--Douady Class of $\cs(\Sigma_c)$}
\label{sec:dixmier-douady-class}

Given a locally compact space $X$, a locally compact abelian group $G$
and a $2$-cocycle $c \in H^2(X, \shh)$, let $\nucp$ be as in
\eqref{eq:10}. We can form the associated Raeburn--Taylor twisted
groupoid $(\gv,\cnuc)$. We want to verify that $\cs(\gv,\cnuc)$ is a
continuous-trace \cs-algebra with Dixmier--Douady class
$[\nucp]=[m_{*}\bigl([c]\bigr)]$, and then to realize $\cs(\gv,\cnuc)$
concretely as $A(\nucp)$.  Unfortunately we cannot refer directly to
\cite{rw:morita}*{Proposition~5.40} for this because $\V$ is not a
cover of $\gu\times X$ by precompact sets and the proof of
\cite{rw:morita}*{Proposition~5.40} assumes this---even if it is not
mentioned in the statement. So in this section we show that
$\cs(\gv,\cnuc)$ is continuous-trace with the desired Dixmier--Douady
class  for the given cover $\V$, and then use this to
prove Theorem~\ref{thm-main-dd-calc}.

We let $A_{1}(\nucp)$ be defined just as in
Section~\ref{sec:raeb-tayl-algebra}.  We want to define
$(f_{ij})(g_{ij})$ to be the matrix $(h_{ij})$ defined by the
equation~\eqref{eq:4}. To see that this determines a binary operation
on $A_{1}(\nucp)$ we need the following analogue of
\cite{rw:morita}*{Proposition~5.39}.
\begin{lemma}\label{lem-prod-ok}
  Let $f=(f_{ij})$ and $g=(g_{ij})$ be elements of $A_{1}(\nucp)$.
  Define
  \begin{equation}\label{eq:8}
    h_{ik}(\tau,x) =\sum_{j} f_{ij}(\tau,x)g_{jk}(\tau,x)
    \overline{\nucp_{ijk}(\tau,x)}.
  \end{equation}
  Then $h_{ik}\in C_{0}(\hG\times X)$ and vanishes off $V_{ik}$.
\end{lemma}
\begin{proof}
  Clearly, $h_{ik}$ vanishes off $V_{ik}$.  Since each point $x$ in
  the compact set $\overline{U_{ij}}$ has a neighborhood $W$ that
  meets only finitely many $U_{l}$, there is a finite set $F$ such
  that
  \begin{equation}
    h_{ik}(\tau,x)=\sum_{j\in F}  f_{ij}(\tau,x)g_{jk}(\tau,x)
    \overline{\nucp_{ijk}(\tau,x)}
  \end{equation}
  for all $(\tau,x)\in V_{ij}=\hG\times U_{ij}$.

  Since each summand is in $C_{0}(\hG\times X)$, it 
  will suffice to see that $h_{ik}$ is continuous on $\hG\times
  X$. Suppose that $(\tau_{n},x_{n})\to (\tau_{0},x_{0})$.  It is
  enough to show that
  $h_{ik}(\tau_{n},x_{n})\to h_{ik}(\tau_{0},x_{0})$.  For this, it
  suffices to consider each summand
  \begin{equation}
    a_{j}(\tau,x)=
    \begin{cases}
      f_{ij}(\tau,x)g_{jk}(\tau,x)
      \overline{\nucp_{ijk}(\tau,x)}&\text{if
        $x\in U_{ijk}$}\\
      0&\text{otherwise.}
    \end{cases}
  \end{equation}
  We clearly have $h_{ik}(\tau_n, x_n) \to h_{ik}(\tau_0, x_0)$ if
  $(\tau_0, x_0) \in V_{ijk}$ or if
  $(\tau_0, x_0) \not\in \overline{V_{ijk}}$. So we suppose that
  $(\tau_{0},x_{0}) \in \overline {V_{ijk}}\setminus V_{ijk}=\hG\times
  (\overline{U_{ijk}}\setminus U_{ijk})$.

  We suppose that
  $a_{j}(\tau_{n},x_{n}) \not\to a_{j}(\tau_{0},x_{0})=0$ and derive a
  contradiction. By passing to a subsequence, we can assume that
  $|a_{j}(\tau_{n},x_{n})|\ge \epsilon > 0$ for all $n$.  Since
  $x_{0}\notin U_{ijk}$, we can assume by symmetry that
  $x\notin U_{ij}$ and hence that
  $x_{0}\in \overline{U_{ij}}\setminus U_{ij}$.  Since $f_{ij}$ is
  continuous on $X$ and vanishes off $U_{ij}$, we have
  $f_{ij}(x_{n})\to 0$. Hence $a_{j}(\tau_{n},x_{n})\to 0$, a
  contradiction.
\end{proof}

Since there is no difficulty with the involution as defined in
\eqref{eq:3}, we see that $A_{1}(\nucp)$ is a $*$-algebra just as in
Section~\ref{sec:raeb-tayl-algebra}.  We get seminorms
$\|f\|_{(\tau,x)} =\|\pi_{(i,(\tau,x))}(f)\|$ almost exactly as in
\cite{rw:morita}*{Proposition~5.40}, and let
\[
  A(\nucp) := \{f\in A_{1}(\nucp) : (\tau,x)\mapsto
  \|f\|_{(\tau,x)}\text{ vanishes at infinity}\}.
\]
Just as in the proof of \cite{rw:morita}*{Proposition~5.40},
$A(\nucp)$ is complete with respect to the norm
\begin{align}\SwapAboveDisplaySkip
  \|f\|=\sup_{(\tau,x)}\|f\|_{(\tau,x)},
\end{align}
and has spectrum identified (topologically) with $\hG\times X$.
Moreover, we have the following analogue of the Raeburn--Taylor
result.
\begin{lemma}
  \label{lem-prop-5.50+} The \cs-algebra $A(\nucp)$ has continuous
  trace with spectrum $\hG\times X$ and Dixmier-Douady class
  $\delta(A(\nucp))=[\nucp]$.
\end{lemma}
\begin{proof}
  We have already seen that $A(\nucp)$ is a \cs-algebra with Hausdorff
  spectrum. We continue by making the necessary modifications to the
  proof of \cite{rw:morita}*{Proposition~5.40}.  Let $\sset{O_{n}}$ be
  a locally finite cover of $\hG$ with open, precompact sets.  To show
  that $A(\nucp)$ has continuous trace, we show that it has local
  rank-one projections as required by
  \cite{rw:morita}*{Definition~5.13}.  Let $(\tau,x)\in\hG\times X$,
  say with $(\tau,x)\in O_{n}\times U_{i}$.  Let
  $\phi\in C_{c}^{+}(O_{n}\times U_{i})$ be such that $\phi\equiv1$
  near $(\tau,x)$.  Let
  \begin{equation}
    \label{eq:18}
    p_{jk}(x) =
    \begin{cases}
      \phi(x)&\text{if $j=i=k$, and} \\ 0 &\text{otherwise.}
    \end{cases}
  \end{equation}
  Then $p=(p_{jk})\in A(\nucp)$ and
  \begin{equation}
    \label{eq:19}
    \pi_{(i,(\tau',x'))}(p)
  \end{equation}
  is a rank-one projection for $(\tau',x')$ near
  $(\tau,x)$.
  This suffices.

  We will calculate $\delta(A(\nucp))$ using
  \cite{rw:morita}*{Lemma~5.28}.  Using the shrinking lemma (see
  \cite{rw:morita}*{Lemma~4.32}), we can find compact sets $F_{n,i}$
  in $W_{n,i}:= O_{n}\times U_{i}$ such that the interiors of the
  $F_{n,i}$ cover $\hG\times X$.  Since $\W=\sset{W_{n,i}}$ is a
  refinement of $\V$, $[\nu]$ is represented by the cocycle in
  $Z^{2}(\W,\uT)$ given by
  \begin{equation}
    \label{eq:20}
    \tilde \nucp_{(n,i)(m,j)(l,k)}(\tau,x)=\nucp_{ijk}(\tau,x).
  \end{equation}

  Let $\phi_{n,i}\in C_{c}^{+}(W_{n,i})$ be such that
  $\phi_{n,i}\equiv 1$ on $F_{n,i}$.  Then as above, we get
  $p(n,i)\in A(\nucp)$ such that
  \begin{equation}
    \label{eq:21}
    \pi_{(i,(\tau,x))}(p(n,i))
  \end{equation}
  is a rank-one projection for all $(\tau,x)\in F_{n,i}$.  Similarly,
  let $\phi_{(n,i)(m,j)} \in C_{c}^{+}(W_{(n,i)(m,j)})$ be identically
  one on $F_{(n,i)(m,j)}$.  Then we get $v((n,i),(m,j))\in A(\nucp)$
  with
  \begin{equation}
    \label{eq:22}
    v((n,i),(m,j))_{rs}(\tau,x)=
    \begin{cases}
      \phi_{(n,i)(m,j)}(\tau,x)&\text{if $r=i$ and $s=j$, and} \\ 0
      &\text{otherwise.}
    \end{cases}
  \end{equation}
  Then check that
  \begin{equation}
    \label{eq:23}
    \pi_{(i,(\tau,x))} \bigl(v((n,i),(m,j))v((n,i),(m,j))^{*}\bigr) =
    \pi_{(i,(\tau,x))} \bigl(p(n,i)\bigr),
  \end{equation}
  while
  \begin{equation}
    \label{eq:24}
    \pi_{(i,(\tau,x))} \bigl(v((n,i),(m,j))^{*}v((n,i),(m,j))\bigr) =
    \pi_{(i,(\tau,x))} \bigl(p(m,j)\bigr)
  \end{equation}
  for all $(\tau,x)\in F_{(n,i)(m,j)}$.

  Note that the situation is symmetric as $\pi_{(j,(\tau,x))}$ is
  equivalent to $\pi_{(i,(\tau,x))}$ if $x\in U_{ij}$.

  If $(\tau,x)\in F_{(n,i)(m,j)(l,k)}$, then
  \begin{equation}
    \label{eq:25}
    \phi_{(n,i)(m,j)}(\tau,x)=\phi_{(m,j),(l,k)}(\tau,x)=
    \phi_{(n,i),(l,k)}(\tau,x)=1.   
  \end{equation}
  Thus{\allowdisplaybreaks
    \begin{align}
      \label{eq:26}
      \bigl[v((n,i),
      &(m,j))v((m,j),(l,k))\bigr]_{rs}(\tau,x) \\
      &= \sum_{a} \overline{\nucp_{rsa}(\tau,x)}
        v((n,i),(m,j))_{ra}(\tau,x) v ((m,j),(l,k))_{as}(\tau,x) \\
      &=
        \begin{cases}
          0&\text{if $r\not=i$ or $s\not=l$, and } \\
          \overline{\nucp_{ijk}(\tau,x)}\phi_{(n,i)(m,j)}(\tau,x)
          \phi_{(m,j)(l,k)}(x,\tau) & \text{if $r=i$ and $s=l$.}
        \end{cases} \\
      &=
        \begin{cases}
          0&\text{if $r\not=i$ or $s\not=l$, and } \\
          \overline{\nucp_{ijk}(\tau,x)}\phi_{(n,i)(l,k)}(\tau,x) &
          \text{if $r=i$ and $s=l$.}
        \end{cases}\\
      &= \overline{\nucp_{ijk}(\tau,x)} v((n,i),(l,k))_{rs}(\tau,x).\\
      &=\overline{\tilde\nucp_{(n,i)(m,j)(l,k)}(\tau,x)}
        v((n,i),(l,k))(\tau,x).
    \end{align}
    This completes the proof.}
\end{proof}
With these modifications in place, we can prove
Theorem~\ref{thm-main-dd-calc}.
\begin{proof}[Proof of Theorem~\ref{thm-main-dd-calc}]
  By Lemma~\ref{lem-prop-5.50+}, it suffices to produce an isomorphism
  $\Phi:\cs(\Sigma_{c})\to A(\nucp)$ that intertwines each
  $\indsu((i,x),\tau)$ with $\pi_{(i, (\tau,x))}$. We use the Fourier
  Transform.  If $f\in C_{c}(\Sigma_{c})$, then we define
  \begin{equation}
    \Phi(f)(i,(\tau,x),j)=\int_{G} \tau(g) f\bigl(g,(i,x,j)\bigr)
    \,d\mu(g) .
  \end{equation}
  To see that $\Phi(f)\in A(\nucp)$, first suppose that there exist
  $\phi\in C_{c}(G)$ and $h\in C_{c}(\gu)$ such that
  $f\bigl( g,(i,x,j)\bigr)=\phi(g)h(i,x,j)$. Then
  \begin{equation}
    \Phi(f)(i,(\tau,x),j)=\hat\phi(\tau)h(i,x,j),
  \end{equation}
  and $\Phi(f)\in A(\nucp)$ because $\hat \phi\in C_{0}(\hG)$. Since
  finite sums of such functions are dense in the inductive limit
  topology, we deduce that $\Phi(f)\in A(\nucp)$ for all $f$.

  Note
  \begin{align}
    \label{eq:43}
    \Phi(f^{*}(i,(\tau,x),j))
    &= \int_{G}\tau(g)
      \overline{f(-g-c_{iji}(x),(j,x,i))}
      \,d\mu(g) \\
    &= \overline{\int_{G} \tau(g) f(g-c_{iji}(x),(j,x,i))\,d\mu(g)} \\
    &= \overline{\tau(c_{iji}(x))} \overline{\int_{G} \tau(g)
      f(g,(j,x,i))\,d\mu(g)} \\
    &=\nucp_{iji}(\tau,x) \overline{\Phi(f)(j,(\tau,x),i))}
    \\
    &=\Phi(f)^{*}(i,(\tau,x),j).
  \end{align}
  Hence $\Phi$ is $*$-preserving.
  
  To see that it is multiplicative, we use~\eqref{eq:11} at the second
  equality, and Fubini's Theorem and invariance of Haar measure at the
  third to calculate:
  \begin{align}
    \Phi(f&*f')(i,(\tau,x),j)\\
          &= \int_{G}\tau(g)f*f'\bigl(g,(i,x,j)\bigr) \,d\mu(g) \\
          &=\sum_{k}\int_{G}\int_{G} \tau(g)f\bigl(h,(i,x,k)\bigr)
            f'\bigl(g-h-c_{ikj}(x), (k,x,j)\bigr) \,d\mu(h)\,d\mu(g) \\
          &=\sum_{k}\int_{G}\int_{G} \tau(g+h+c_{ikj}(x))f\bigl(h,(i,x,k)\bigr)
            f'\bigl(g, (k,x,j)\bigr) \,d\mu(g)\,d \mu(h) \\
          &= \sum_{k}
            \tau(c_{ikj}(x))\Phi(f)\bigl(i,(\tau,x),k)\bigr) 
            \Phi(f')\bigl(k, (\tau, x), j)\bigr)\\
          &= \sum_{k} \Phi(f)\bigl(i,(\tau,x),k)\bigr)
            \Phi(f')\bigl(k, (\tau, x), j)\bigr) 
            \cnuc\bigl((i,(\tau,x),k),(k,(\tau,x),j\bigr) \\
          &=\Phi(f)*\Phi(f')\bigl(i,(\tau,x),j\bigr).
  \end{align}

  It remains to see that $\Phi$ is isometric and surjective.
  It follows from Lemma~\ref{lem-rho-reps}, that
  $\indsu((i,x),\tau)(f)$ is equivalent to multiplication by the
  matrix
  \begin{align}\SwapAboveDisplaySkip
    \label{eq:44}
    \bigl[ \tau(c_{ijk})(x)\Phi(f) \bigr].
  \end{align}
  Since $\tau(c_{ijk}(x))=\overline{\nucp_{ijk}(\tau,x)}$, we see that
  \begin{align}
    \label{eq:45}
    \indsu((i,x),\tau)(f)=\pi_{(i,(\tau,x))}(\Phi(f)).
  \end{align}
  
  This shows immediately that $\Phi$ is isometric.  It also shows that
  the image $\Phi(\cs(\Sigma_{c}))$ is a rich subalgebra (in the sense
  of \cite{dix:cs-algebras}*{Definition~11.1.1}) of the
  continuous-trace \cs-algebra $A(\nucp)$.  Hence $\Phi$ is surjective
  by \cite{dix:cs-algebras}*{Proposition~11.1.6}.
\end{proof}

\section{Groupoids Associated to Local Homeomorphisms}
\label{sec:local-lomeomorphisms}

In this section we extend our results from
Section~\ref{sec:nice-class-examples} to a more general setting. Let
$\psi:Y\to X$ be a local homeomorphism and form the principal groupoid
\begin{equation}
  R(\psi)=\set{(x,y)\in Y\times Y:\psi(x)=\psi(y)}.
\end{equation}
Let $G$ be a locally compact abelian group. Given a
unit-space-preserving groupoid extension
\begin{equation}\label{eq:9}
  \begin{tikzcd}[column sep=3cm]
    G\times Y \arrow[r,"\iota", hook] \arrow[dr,shift left, bend right
    = 15] \arrow[dr,shift right, bend right = 15]&\Sigma
    \arrow[r,"\pi", two heads] \arrow[d,shift left] \arrow[d,shift
    right]&R(\psi)\arrow[dl,shift left, bend left = 15]
    \arrow[dl,shift right, bend left = 15]
    \\
    &Y&
  \end{tikzcd}
\end{equation}
such that $\iota(g,r(\sigma))\sigma = \sigma\iota(g,s(\sigma))$ for
all $g\in G$ and $\sigma \in \Sigma$, the groupoid $\Sigma$ is a
groupoid with central isotropy. As described in detail in
Appendix~\ref{sec:Extcoc}, $\Sigma$ is a $G$-twist over $R(\psi)$.  So
$\Sigma$ is a principal $G$-bundle over $R(\psi)$ with $G$ action
\begin{equation}
  g \cdot \sigma =\iota(g,r(\sigma)) \sigma=
  \sigma\iota(g,s(\sigma))=\sigma\cdot g.
\end{equation}
We endow $\Sigma$ with the Haar system $\sset{\lambda^y}$ via
\begin{equation}\label{eq:1}
  \int_{\Sigma}
  f(\sigma)\,d\lambda^y(\sigma):=\sum_{r(\sigma)=y}\int_G  f(g\cdot 
  \sigma)\,d\mu(g) =\sum_{r(\sigma)=y} \int_{G}f(\sigma\cdot
  g)\,d\mu(g). 
\end{equation}

If $\pi$ has a continuous section $\kappa:R(\psi)\to \Sigma$ (this is
equivalent to $\pi$ being trivial as a principal $G$-bundle), then
Proposition~\ref{prop:topological-section} shows that $\Sigma$ is
properly isomorphic to the extension $\gpdtwst{R(\psi)}{\varphi}$
constructed from a continuous (normalised) $G$-valued $2$-cocycle
$\varphi \in Z^{2}(R(\psi),G)$ as in Notation~\ref{not:Esigmadef}.

To proceed, we need to assume that the map $\pi$ in~\eqref{eq:9} has
local sections: that is, for each $(x,y) \in R(\psi)$, there is a
neighbourhood $U$ of $(x,y)$ on which there is a continuous map
$s : U \to \Sigma$ satisfying $\pi \circ s = \id_U$. If $G$ is a Lie
group, then this is automatic due to the Palais Slice Theorem
\cite{pal:aom61}*{\S4.1}.  But in addition, we need to guarantee that
the collection of local sections is sufficiently robust to allow us to
build an equivalent groupoid with a global section. To this end, we
assume that $X$ is \emph{locally $G$-trivial}: every open cover of $X$
has a refinement $\sset{W_{i}}$ such that each
$H^{1}(W_{ij},\uG)=\sset0$. Equivalently, all locally trivial
principle $G$-bundles over the double-overlaps $W_{ij}$ are trivial.
A special case where these assumptions automatically hold is when $G$
is a Lie group and $X$ admits good covers in the sense that every open
over of $X$ admits a refinement in which all nontrivial overlaps are
contractible. This suffices: locally trivial principle $G$-bundles
over a space $Z$ have a classifying space $BG$ so that bundle classes
are parameterized by homotopy classes $[Z,BG]$; so if $Z$ is
contractible, then all bundles over $Z$ are trivial. All
differentiable manifolds admit good covers by
\cite{bottu:diff82}*{Corollary~I.5.2}.

For our main result, we will need to verify that the Morita equivalences
we will use preserve the identification of the spectra with
$\hG\times X$ in each case.  In particular, recall that a subset
$U\subset Y=\go$ is \emph{full} if it meets every orbit: equivalently,
$\Sigma\cdot U=\go$.  It that case, $\Sigma_{U}$ is a
$(\Sigma,\Sigma(U))$-equivalence.  Then induction from
$\cs(\Sigma(U))$ to $\cs(\Sigma)$, $\ibind{\Sigma_{U}}$, induces the
Rieffel homeomorphism of $\cs(\Sigma(U))^{\wedge}$ onto
$\cs(\Sigma)^{\wedge}$.  (For the basics on induced representations in
this context, see \cite{ionwil:pams08}*{\S2}.)  Both these
\cs-algebras have spectrum identified with $\hG\times X$, and the
observation that the Rieffel homeomorphism is the identity with
respect to these identifications follows immediately from the next
lemma.
\begin{lemma}
  \label{lem-full} If $y\in U$ and $\tau\in \hG$, then
  $\ibind{\Sigma_{U}} (\indsu^{\Sigma(U)}(y,\tau))$ is equivalent to
  $\indsu^{\Sigma} (y,\tau)$.
\end{lemma}
\begin{proof}
  As in \cite{mrw:jot87}*{p.~12} or
  \cite{simwil:jot11}*{Equation~1.3}, the
  $C_{c}(\Sigma({U}))$-valued inner product on $C_{c}(\Sigma_{U})$
  is given by
  \begin{align}
    \label{eq:27}
    \rip \Sigma(U)<f_{1},f_{2}>(\gamma)=\int_{\Sigma}
    \overline{f_{1}(\sigma^{-1})} f_{2}(\sigma^{-1}\gamma)
    \,d\lambda^{r(\gamma)}(\sigma). 
  \end{align}
  Then $\ibind{\Sigma_{U}} \indsu^{\Sigma(U)}(y,\tau)$ acts by
  convolution on the completion of
  $C_{c}(\Sigma_{U})\atensor C_{c}(\Sigma(U)_{y})$ with respect to the
  inner product
  \begin{align}
    \label{eq:28}
    \bigl(f_{1}\tensor k_{1}
    & \mid f_{2}\tensor k_{2}\bigr)
      = \bip(\rip \Sigma(U)<f_{2},f_{1}>*k_{1}|k_{2}) \\
    \intertext{which, by \eqref{eq:suip}, is}
    &=\int_{G}\int_{\Sigma(U)} \overline{k_{2}(\sigma^{-1})} \rip
      \Sigma(U)<f_{2} ,f_{1}>*k_{1}(\sigma^{-1}\cdot g) \tau(g)
      \,d\lambda_{U}^{y} (\sigma)\,d\mu(g) \\
    &=\int_{G}\int_{\Sigma(U)}\int_{\Sigma(U)}
      \overline{k_{2}(\sigma^{-1})} \rip \Sigma(U)<f_{2},f_{1}>(\eta)
      k_{1}(\eta^{-1}\sigma^{-1}\cdot g) \tau(g) \\
    &\hskip2in
      \,d\lambda_{U}^{s(\sigma)}(\eta)\tau(g) 
      \,d\lambda_{U}^{y}(\sigma)\,d\mu(g) \\
    \intertext{which, after sending $\eta\mapsto \sigma^{-1}\eta$, is}
    &=\int_{G}\int_{\Sigma(U)}\int_{\Sigma(U)}
      \overline{k_{2}(\sigma^{-1})} \rip \Sigma(U)<f_{2},f_{1}>(\sigma^{-1}\eta)
      k_{1}(\eta^{-1}\cdot g) \tau(g) \\
    &\hskip2in
      \,d\lambda_{U}^{y}(\eta)\tau(g) 
      \,d\lambda_{U}^{y}(\sigma)\,d\mu(g) \\
    &=\int_{G}\int_{\Sigma(U)}\int_{\Sigma(U)}\int_{\Sigma}
      \overline{f_{2}(\gamma^{-1}) k_{2}(\sigma^{-1})}
      f_{1}(\gamma^{-1}\sigma^{-1}\eta) k_{1}(\eta^{-1}\cdot g) \tau(g) \\
    &\hskip2in
      \,d\lambda^{s(\sigma)}(\gamma)
      \,d\lambda_{U}^{y}(\eta)\,d\lambda_{U}^{y}(\sigma) \,d\mu(g) \\
    \intertext{which, after $\gamma\mapsto \sigma^{-1}\gamma$, is}
    &=\int_{G}\int_{\Sigma(U)}\int_{\Sigma(U)}\int_{\Sigma}
      \overline{f_{2}(\gamma^{-1}\sigma) k_{2}(\sigma^{-1})}
      f_{1}(\gamma^{-1}\eta) k_{1}(\eta^{-1}\cdot g) \tau(g) \\
    &\hskip2in
      \,d\lambda^{y}(\gamma)
      \,d\lambda_{U}^{y}(\eta)\,d\lambda_{U}^{y}(\sigma) \,d\mu(g)  \\
    &= \int_{G}\int_{\Sigma} \overline{W(f_{2}\tensor
      k_{2})(\sigma^{-1})}W(f_{1}\tensor k_{1})(\sigma^{-1}\cdot g)
      \,d\lambda^{y}(\sigma) \tau(g) \,d\mu(g) ,
  \end{align}
  where
  \begin{equation}
    \label{eq:29}
    W(f\tensor k)(\sigma):=\int_{\Sigma(U)}
    f(\sigma\eta)k(\eta^{-1})\,d\lambda_{U}^{y} (\eta).
  \end{equation}
  It follows that $W$ defines an isometry from the space of
  $\ibind{\Sigma_{U}}(\indsu^{\Sigma(U)}(y,\tau)$ into the space of
  $\indsu^{\Sigma}(y,\tau)$ that intertwines the two representations.
  Since the representations are irreducible, $W$ must be a
  unitary and the representations must be equivalent.
\end{proof}

We also will need to examine the case of blowing up the unit space
$\go$ with respect to a locally finite cover $\U=\sset{U_{i}}$.  This
gives us the equivalent groupoid
\begin{equation}
  \label{eq:13}
  \Sp=\set{(i,\sigma,j):\text{$\sigma\in\Sigma$, $r(\sigma)\in U_{i}$
      and $s(\sigma)\in U_{j}$,}}
\end{equation}
where the $(\Sigma,\Sp)$-equivalence is given by
\begin{equation}
  \label{eq:30}
  Z:=\coprod \Sigma_{U_{i}}=\set{(i,\sigma):\text{$\sigma\in\Sigma$
      and $s(\sigma)\in U_{i}$.}}
\end{equation}
We endow $\Sp$ with a Haar system $\lambdap=\sset{\lambdap^{(i,x)}} $
just as in \eqref{eq:1}. 
As above, every irreducible representation of $\cs(\Sp)$ is equivalent
to one of the form $\indsu^{\Sp}((i,y),\tau)$ for $y\in\go$ and
$\tau\in\hG$.  As before, we want the Rieffel homeomorphism induced by
$\ibind Z$ to preserve the identification of these spectra with
$\hG\times X$.  This is verified in the next lemma which is analogous
to Lemma~\ref{lem-full}.
\begin{lemma}
  \label{lem-blow-up} With $\U$ and $\Sp$ as above, we have
  $\ibind Z (\indsu^{\Sp}((i,y),\tau)$ equivalent to
  $\indsu^{\Sigma}(y,\tau)$ for all $y\in U_{i}$ and $\tau\in\hG$.
\end{lemma}
\begin{proof}
  The $C_{c}(\Sp)$-valued inner product on $C_{c}(Z)$ is given by
  \begin{align}
    \label{eq:31}
    \rip\Sp<f_{1},f_{2}>(i,\gamma,j)=\int_{\Sigma}
    \overline{f_{1}(i,\sigma^{-1})}
    f_{2}(j,\sigma^{-1}\gamma)\,d\lambda^{r(\gamma)} (\sigma).
  \end{align}
{\allowdisplaybreaks Thus $\ibind Z
  (\indsu^{\Sp}((i,y),\tau)$ acts by convolution on the completion of
  $C_{c}(Z)\atensor C_{c}(\Sp_{(i,y)})$ with respect to the inner product
  \begin{align}
    \label{eq:32}
    \bigl(f_{1}\tensor k_{1}
    & \mid f_{2}\tensor k_{2}\bigr)
                               = \bip(\rip\Sp<f_{2},f_{1}>*k_{1}|
      k_{2})
      \intertext{which, in view of \eqref{eq:suip}, is}
      &= \int_{G}\int_{\Sp} \overline{k_{2}(j,\sigma^{-1},j)}
        \rip\Sp<f_{2},f_{1}>*k_{1}(j,\sigma^{-1}\cdot g,i) \tau(g)
        \,d\lambdap^{(i,y)} (i,\sigma,j) \,d\mu(g) \\
    &= \int_{G}\int_{\Sp}\int_{\Sp} \overline{k_{2}(j,\sigma^{-1},j)}
      \rip\Sp<f_{2}, f_{1}>(j,\eta,l)
      k_{1}(l,\eta^{-1}\sigma^{-1}\cdot g,i)
      \tau(g) \\
    &\hskip2in \,d\lambdap^{(j,s(\sigma)}(j,\eta,l)
      \,d\lambdap^{(i,y)}(i,\sigma,j) \,d\mu(g) \\
    \intertext{which, after invoking left-invariance, is}
    &\int_{G}\int_{\Sp}\int_{\Sp}  \overline{k_{2}(j,\sigma^{-1},j)}
      \rip\Sp<f_{2},f_{1}>(j,\sigma^{-1}\eta,l) k_{1}(l,\eta^{-1}\cdot
      g,i) \tau(g)\\
    &\hskip2in \,d\lambdap^{(i,y)}(i,\eta,l)
      \,d\lambdap^{(i,y)}(i,\sigma,j) \,d\mu(g) \\
    &= \int_{G}\int_{\Sp}\int_{\Sp}\int_{\Sigma}
      \overline{k_{2}(j,\sigma^{-1},j)}
      \overline{f_{2}(j,\gamma^{-1})}
      f_{1}(l,\gamma^{-1}\sigma^{-1}\eta) k_{1}(l,\eta^{-1}\cdot g,i)
      \tau(g) \\
    &\hskip2in \,d\lambda^{s(\sigma)}(\gamma)
      \,d\lambdap^{(i,y)}(i,\eta,l) \,d\lambdap^{(i,y)}(i,\sigma,j)
      \,d\mu(g) \\
    &= \int_{G}\int_{\Sp}\int_{\Sp}\int_{\Sigma}
      \overline{k_{2}(j,\sigma^{-1},j)}
      \overline{f_{2}(j,\gamma^{-1}\sigma)} f_{1}(l,\gamma^{-1}\eta)
      k_{1}(l,\eta^{-1} \cdot g,i)\tau(g) \\
    &\hskip2in \,d\lambda^{y}(\gamma)
      \,d\lambdap^{(i,y)}(i,\eta,l) \,d\lambdap^{(i,y)}(i,\sigma,j)
      \,d\mu(g) \\
    &=\int_{G}\int_{\Sigma}\overline{W(f_{2}\tensor
      k_{2})(\gamma^{-1})} W(f_{1}\tensor k_{1})(\gamma^{-1}\cdot g)
      \,d\lambda^{y}(\gamma)  \tau(g)  \,d\mu(g),
  \end{align}
  where}
\begin{equation}
  \label{eq:33}
  W(f\tensor k)(\gamma) =\int_{\Sp}f\bigl((i,\gamma)\cdot
  (i,\sigma,j)\bigr) k(j,\sigma^{-1},i)\,d\lambdap^{(i,y)}(i,\sigma,j).
\end{equation}
As in the proof of Lemma~\ref{lem-full}, $W$ extends  to an
intertwining 
unitary implementing the desired equivalence.
\end{proof}

\begin{thm}
  \label{thm-dd-class-local-homeo} Let $Y$ and $X$ be second-countable
  locally compact Hausdorff spaces with $X$ locally $G$-trivial as
  defined above.  Suppose that $\psi:Y\to X$ is a local homeomorphism,
  and let $\Sigma$ be  a groupoid extension as in
  \eqref{eq:9}.  Then $\cs(\Sigma)$ has continuous trace with
  spectrum identified with $\hG\times X$ via
  $(\tau,x)\mapsto \indsu((i,y),\tau)$ for any $y\in\Phi^{-1}(x)$.
  Furthermore, there is a locally finite open covering
  $\mathscr W=\sset{W_{j}}_{j\in J}$ of $X$ and a cocycle
  $c \in Z^{2}(\mathscr W,\uG)$ (given in \eqref{eq:12} below) such
  that the Dixmier--Douady invariant of $C^*(\Sigma)$ is given by the
  image of $m_{*}([c])$ in $H^{3}(\hG\times X,\Z)$.
\end{thm}

\begin{proof}
  Let $\sset{U_{i}}$ be a family of open subsets of $Y$ such that each
  $\psi|_{U_{i}}$ is injective, and the sets $\sset{\psi(U_{i})}$
  cover $X$. (For example, any cover $\{U_i\}$ of $Y$ by sets on which
  $\psi$ is injective.) Since $X$ is locally $G$-trivial, there is a
  locally finite refinement $\mathscr W=\sset{W_{j}}_{j\in J}$ of the
  cover $\{\psi(U_i)\}$ of $X$ such that $H^{1}(W_{ij},\uG)=\sset0$
  for all $i$ and $j$. Fix $r:J\to I$ such that each
  $W_{j}\subset \psi(U_{r(j)})$. For each $j$, let
  $V_{j}= \psi^{-1}(W_{j})\cap U_{r(j)}$ so that $\phi$ restricts to a
  homeomorphism of $V_{j}$ onto $W_{j}$.  Let
  $Y'=\bigcup V_{j} \subseteq Y$. Then $Y'$ is open and meets every
  orbit in $Y$. Hence $\ssyp$ is equivalent to $\Sigma$  and we can apply
  Lemma~\ref{lem-full}.  (The $\Sigma$-orbits and $R(\psi)$-orbits
  coincide on
  $Y$.) 
  We can blow up $\ssyp$ with respect to the cover $\sset{V_{j}}$ to
    get an equivalent groupoid
    \begin{equation}
      \label{eq:34}
      \Sigma' =
      \set{(i,\sigma,j):\text{$\sigma \in \Sigma, r(\sigma)\in V_{i}$
          and $s(\sigma)\in V_{j}$}}
    \end{equation}
     and then apply Lemma~\ref{lem-blow-up}.  If we let
    \begin{equation}
      \label{eq:35}
      R' = \set{(i,(x,y),j):\text{$\psi(x)=\psi(y)$, $x\in V_{i}$ and
          $y\in V_{j}$}},
    \end{equation}
    then we obtain a generalised twist
    \begin{equation}
      \begin{tikzcd}[column sep=3cm]
        G \times \coprod V_{j} \arrow[r,"\iota'", hook]
        \arrow[dr,shift left, bend right = 15] \arrow[dr,shift right,
        bend right = 15]&\Sigma' \arrow[r,"\pi'", two heads]
        \arrow[d,shift left] \arrow[d,shift right]&R'.\arrow[dl,shift
        left, bend left = 15] \arrow[dl,shift right, bend left = 15]
        \\
        &\coprod V_{j}&
      \end{tikzcd}
    \end{equation}

  The map $(x,y)\mapsto \psi(x)$ is a homeomorphism of
  $R(\psi)\cap (V_{i} \times V_{j})$ onto $W_{ij}$. Hence
  ${\pi^{-1}\big(R(\psi)\cap (V_{i}\times V_{j})}\big)$ is a trivial
  bundle by our assumption on the $\sset{W_{i}}$ and there is a local
  section $\kappa_{ij}$ defined on $R(\psi)\cap (V_{i} \times V_{j})$
  (we may choose $\kappa_{ii}$ to respect the identification of unit
  spaces). By definition of $R'$ and $\Sigma'$ each $\kappa_{ij}$
  determines a section
  $\kappa'_{ij} : \{i\} \times (R(\psi)\cap (V_{i} \times V_{j}))
  \times \{j\} \to \Sigma'$ satisfying
  \begin{equation}
    \kappa'_{ij}(i,(x,y),j)=(i,\kappa_{ij}(x,y),j).
  \end{equation}
  Since the domains of the $\kappa'_{ij}$ are topologically disjoint
  in $R'$, these $\kappa'_{ij}$ assemble into a global section
  $\kappa' : R' \to \Sigma'$ for $\pi'$.  Hence there is a continuous
  $G$-valued normalised cocycle $\varphi \in Z^{2}(R',G)$ such that
  $\Sigma'\cong \gpdtwst{R'}{\varphi}$ (see
  Proposition~\ref{prop:topological-section} for details). There is a
  groupoid homomorphism $\tau:R'\to \Gamma_{\mathscr W}$ such that
  \begin{equation}
    \tau(i,(x,y),j)=(i,\psi(x),j)
  \end{equation}
  and
  \begin{equation}
    \tau^{-1}(i,w,j)=(i,(x,y),j)\quad\text{if $x\in V_{i}$ and $y\in
      V_{j}$ satisfy $\psi(x) = w = \psi(y)$.}
  \end{equation}

  So, defining
  $\tilde{\varphi} := \varphi \circ (\tau^{-1} \times \tau^{-1}) \in
  Z^{2}(\Gamma_{\mathscr W},G)$, we obtain an isomorphism
  $\Sigma'\cong \gpdtwst{\Gamma_{\mathscr W}}{\tilde{\varphi}}$.  We
  define $c\in Z^{2}(\mathscr W,\uG)$ by
  \begin{equation}\label{eq:12}
    c_{ijk}(w)=\tilde \varphi((i,w,j),(j,w,k)).
  \end{equation}
  Then $\gpdtwst{\Gamma_{\mathscr W}}{\tilde{\varphi}} = \Sigma_c$.

  The isomorphism of $\Sp$ and $\Sigma_{c}$ clearly intertwines the
  two Haar systems (see Proposition~\ref{prop:topological-section}),
  and therefore the representations $\indsu^{\Sp}((i,\psi(x))$ and
  $\indsu^{\Sigma_{c}}((i,x),\tau)$.  Combining this with
  Lemmas~\ref{lem-blow-up}~and \ref{lem-full}, the Dixmier-Douady
  class of $\cs(\Sigma)$ can be identified with that of
  $\cs(\Sigma_{c})$.  The result now follows from
  Theorem~\ref{thm-main-dd-calc}.
\end{proof}

\appendix

\section{Extensions and Cocycles}
\label{sec:Extcoc}

Let $G$ be a locally compact Hausdorff abelian group and let $\Gamma$
be a locally compact Hausdorff groupoid with Haar system
$\{\lambda^{u}\}_{u\in\Gamma^{(0)}}$. Following
\citelist{\cite{kum:jot88} \cite{tu:tams06} \cite{ionkum16}}, we
define an \emph{extension} (or twist) by $G$ over $\Gamma$ to be a
central groupoid extension
$\Gamma^{(0)}\times G\xrightarrow{\iota}\Sigma\xrightarrow{\pi}\Gamma$
where $\go=\Gamma^{(0)}$, $\iota$ is a groupoid homeomorphism onto a
closed subgroupoid of $\Sigma$ such that $\iota(u,0_{G})=u$ for all
$u\in\Gamma^{(0)}$, $\pi$ is an open, surjective groupoid homomorphism
such that $\pi(u)=u$ for all $u\in\Gamma^{(0)}$, such that
$\pi^{-1}(\Gamma^{(0)})=\iota(\Gamma^{(0)}\times G)$, and such that
$\iota(r(\sigma),g)\sigma=\sigma \iota(s(\sigma),g)$ for all
$\sigma\in\Sigma$ and $g\in G$. We summarize all of this by drawing
the diagram
\[
  \begin{tikzcd}[column sep=3cm]
    \Gamma^{(0)} \times G \arrow[r,"\iota", hook] \arrow[dr,shift
    left, bend right = 15] \arrow[dr,shift right, bend right =
    15]&\Sigma \arrow[r,"\pi", two heads] \arrow[d,shift left]
    \arrow[d,shift right]&\Gamma \arrow[dl,shift left, bend left = 15]
    \arrow[dl,shift right, bend left = 15]
    \\
    &\go.&
  \end{tikzcd}
\]
Two twists by $G$ are \emph{properly isomorphic} if there is a
groupoid isomorphism between them which preserves the inclusions of
$\Gamma^{(0)}\times G$ and intertwines the surjections onto
$\Gamma$. If
$\Gamma^{(0)}\times
G\stackrel{\iota}{\longrightarrow}\Sigma\stackrel{\pi}{\longrightarrow}\Gamma$
is an extension by $G$ over $\Gamma$, then $G$ acts freely and
properly on $\Sigma$ via $g\sigma:=\iota(r(\sigma),g)\sigma$. Hence
$\pi:\Sigma\to\Gamma$ is a principal $G$-bundle. Moreover,
$G\setminus\Sigma\simeq\Gamma$ and $\pi$ can be identified with the
quotient map.

\begin{remark}
  For completeness, we check that the action of $G$ on $\Sigma$ is
  proper. We need to show that the map
  $(g,\sigma)\mapsto(g\sigma,\sigma)$ is proper. Let $K$ be a compact
  subset of $\Sigma$ and let $\{(g_{n},\sigma_{n})\}_{n}$ be a
  sequence in the preimage of $K\times K$. Then
  $\{g_{n}\sigma_{n}\}_{n}\subset K$ and
  $\{\sigma_{n}\}_{n}\subset K$.  Hence there is a subsequence
  $\{\sigma_{n_{k}}\}_{k\ge1}$ such that
  $\sigma_{k_{n}}\to\sigma\in K$ and
  $g_{k_{n}}\sigma_{k_{n}}\to\sigma^{\prime}\in K$. It follows that
  $\pi(\sigma)=\pi(\sigma^{\prime}).$ Hence there exists $g\in G$ such
  that $\sigma^{\prime}=g\sigma$. Therefore
  \[
    \iota(r(\sigma_{k_{n}}),g_{k_{n}})=\iota(r(\sigma_{k_{n}}),g_{k_{n}})\sigma_{k_{n}}\sigma_{k_{n}}^{-1}\to
    g\sigma\sigma^{-1}=\iota(r(\sigma),g).
  \]
  Since $\iota$ is a homeomorphism, it follows that $g_{k_{n}}\to g$,
  so the action of $G$ on $\Sigma$ is proper.

\end{remark}
As in \cite{tu:tams06}, the Baer sum $\Sigma_{1}\ast\Sigma_{2}$ of two
extensions
$\Gamma^{(0)}\times G\longrightarrow\Sigma_{i}\longrightarrow\Gamma$
is the extension
$\Gamma^{(0)}\times G\longrightarrow\Sigma_{1}\ast
\Sigma_{2}\longrightarrow\Gamma$ with
\[
  \Sigma_{1}\ast
  \Sigma_{2}=\{(\sigma_{1},\sigma_{2})\in\Sigma_{1}\times\Sigma_{2}\:|\:\pi_{1}(\sigma_{1})=\pi_{2}(\sigma_{2})\}/\sim,
\]
where $(g\sigma_{1},\sigma_{2})\sim(\sigma_{1},g\sigma_{2})$. The map
$\pi:\Sigma\to G$ is given by
$\pi[(\sigma_{1},\sigma_{2})]=\pi_{1}(\sigma_{1})=\pi_{2}(\sigma_{2})$,
and the inclusion $\iota:\Gamma^{(0)}\times G\to\Sigma$ is
$\iota(u,g)=[(\iota_{1}(u,g),u)]=[(u,\iota_{2}(u,g))]$. The inverse of
the extension
$\Gamma^{(0)}\times G\xrightarrow{\iota}\Sigma\xrightarrow{\pi}\Gamma$
is the extension
$\Gamma^{(0)}\times
G\xrightarrow{\iota^{\prime}}\tilde{\Sigma}\xrightarrow{\pi}\Gamma$,
where $\tilde{\Sigma}=\Sigma$ as a groupoid, but
$\iota^{\prime}(g)=\iota(-g)$. The semi-direct product
$\Gamma\times G$ is called the \emph{trivial twist.} The collection
$T_{\Gamma}(G)$ of proper isomorphism classes of twists by $G$ forms
an abelian group under $*$ with neutral element $[\Gamma\times G]$.

\begin{remark}
  Let $\Sigma$ be an extension by $G$ over a principal groupoid
  $\Gamma$; that is, the map $\gamma\mapsto(r(\gamma),s(\gamma))$ is
  injective; equivalently, $\I(\Gamma) = \go$.  Then
  $\iota(\go\times G)= \I(\Sigma)$:
  $\iota(\go\times G)\subseteq \I(\Sigma)$ is clear; and if
  $\sigma\in \I(\Sigma)$ then $\pi(\sigma)\in \I(\Gamma)=\go$ and,
  hence $\sigma\in \iota(\go\times G).$ As in the case of
  $\mathbf{T}$-groupoids, if in addition $\Sigma\setminus\go$ is
  Hausdorff, then $\Sigma$ has central isotropy.
\end{remark}
A continuous 2-cocycle $\varphi:\Gamma^{(2)}\to G$ is a continuous
function such that for all
$(\gamma_{0},\gamma_{1},\gamma_{2})\in\Gamma^{(3)}$,
\[
  \varphi(\gamma_{0},\gamma_{1}) +
  \varphi(\gamma_{0}\gamma_{1},\gamma_{2}) =
  \varphi(\gamma_{1},\gamma_{2}) +
  \varphi(\gamma_{0},\gamma_{1}\gamma_{2}).
\]
A 2-cocycle $\varphi$ is \emph{normalised }provided that
$\varphi(u,\gamma)=0 = \varphi(\gamma,u)$ for all
$\gamma \in \Gamma, u \in \go$. The collection of all normalized
cocycles forms a group denoted $Z^{2}(\Gamma, G)$.  A $1$-cochain is a
continuous function $f : \Gamma \to G$.
The associated $2$-coboundary is the map
$d^1f \in Z^2(\Gamma, G)$ given by
$(d^{1}f)(\gamma_{0}, \gamma_{1}) = f(\gamma_{0}) + f(\gamma_{1}) -
f(\gamma_{0}\gamma_{1})$.

\begin{remark}
  If $\Gamma$ is \'etale, then any continuous 2-cocycle is
  cohomologous to a normalized continuous 2-cocycle. Indeed, note
  first that $\varphi(\gamma_{0},u) = \varphi(u,\gamma_{1})$ for all
  $(\gamma_{0},u,\gamma_{1})\in\Gamma^{(3)}$. Define a $1$-cochain
  $f:\Gamma\to G$ by $f(\gamma) = 0_{G}$ for all $\gamma\notin\go$ and
  $f(u) = \varphi(u,u)$ for $u\in\go$. Since $\Gamma$ is \'etale,
  $\go$ is open so $f$ is continuous. So
  $\psi(\gamma_{0},\gamma_{1}) := \varphi(\gamma_{0}, \gamma_{1}) -
  (d^{1}f)(\gamma_{0}, \gamma_{1})$ defines a continuous normalized
  $2$-cocycle cohomologous to $\varphi$.
\end{remark}

\begin{notation}\label{not:Esigmadef}
  Recall from \cite{ren:groupoid}*{Lemma I.1.14} that given a
  normalized 2-cocycle $\varphi$ on $\Gamma$, there is an extension
  $\gpdtwst{\Gamma}{\varphi}$ of $\Gamma$ by $G$ given by
  \[
    \gpdtwst{\Gamma}{\varphi} := G\times\Gamma
  \]
  with $(g,\gamma_{1})$ and $(h,\gamma_{2})$ composable if and only if
  $\gamma_{1}$ and $\gamma_{2}$ are composable, and
  $(g,\gamma_{1}) (h,\gamma_{2}) = (g + h + \varphi(\gamma_{1},
  \gamma_{2}), \gamma_{1}\gamma_{2})$. We have
  $(g,\gamma)^{-1} = (-g - \varphi(\gamma^{-1}, \gamma),
  \gamma^{-1})$. The maps
  $\iota_{\varphi} : \Gamma^{(0)} \times G \to
  \gpdtwst{\Gamma}{\varphi}$ and
  $\pi_{\varphi} : \gpdtwst{\Gamma}{\varphi} \to \Gamma$ are defined
  by $\iota_{\varphi}(u,g) = (g,u)$ and
  $\pi_{\varphi}(g, \gamma) = \gamma$.
\end{notation}

\begin{remark}
  One can prove that the proper-isomorphism class of
  $\gpdtwst{\Gamma}{\varphi}$ depends only on the cohomology class of
  $\varphi$. Indeed, if $\varphi_{2}=\varphi_{1}+d^{1}f$, then the map
  $\psi: \gpdtwst{\Gamma}{\varphi_1} \to \gpdtwst{\Gamma}{\varphi_2}$
  defined via $\psi(g,\gamma)=(g-f(\gamma),\gamma)$ is a proper
  isomorphism.
\end{remark}
The following result generalizes the discussion on pages 130--131 of
\cite{muhwil:ms92} (see also \cite{ren:groupoid}*{Lemma I.1.14}).

\begin{prop}\label{prop:topological-section}
  An extension $\Sigma$ is properly isomorphic to
  $\gpdtwst{\Gamma}{\varphi}$ for some continuous normalized
  $2$-cocycle $\varphi \in Z^{2}(\Gamma, G)$ if and only if the map
  $\pi$ admits a continuous cross section $\tau$.
\end{prop}
\begin{proof}
  Assume that $i: \gpdtwst{\Gamma}{\varphi} \to\Sigma$ is a proper
  isomorphism, where $\varphi$ is a continuous 2-cocycle. Then one can
  define a continuous cross section $\tau$ of $\pi$ by
  $\tau(\gamma)=i(0_{G},\gamma)$. Conversely, assume that
  $\tau:\Gamma\to\Sigma$ is a continuous cross section of $\pi$. By
  replacing $\tau$ with the map
  $\gamma \mapsto \tau(r(\gamma))^{-1}\tau(\gamma)$, we can assume
  without loss of generality that $\tau(u)=u$ for all $u\in\go$. Then
  $\tau(\gamma_{1})\tau(\gamma_{2})\tau(\gamma_{1}
  \gamma_{2})^{-1}\in\pi^{-1}(\Gamma^{(0)})=\iota(\Gamma^{(0)}\times 
  G)$ for all $(\gamma_{1},\gamma_{2})\in\Gamma^{(2)}$. Since $\iota$
  is a homeomorphism onto its image, there is a unique
  $(u,g)\in\Gamma^{(0)}\times G$ such that
  $\iota(u,g)=\tau(\gamma_{1})
  \tau(\gamma_{2})\tau(\gamma_{1}\gamma_{2})^{-1}$. Note 
  that $u=r(\gamma_{1})$. Define $\varphi:\Gamma^{(2)}\to G$ by
  $\iota(r(\gamma_{1}),\varphi(\gamma_{1},\gamma_{2})) =
  \tau(\gamma_{1})\tau(\gamma_{2})\tau(\gamma_{1}\gamma_{2})^{-1}$. Then
  $\varphi$ is continuous. To see that $\varphi$ is a 2-cocycle, first
  note that
  $\iota(r(\gamma_{1}), \varphi(\gamma_{1}, \gamma_{2}))
  \tau(\gamma_{1}, \gamma_{2}) = \tau(\gamma_{1}) \tau(\gamma_{2})$
  and that $\Sigma$ is an extension. So if
  $(\gamma_{0},\gamma_{1},\gamma_{2})\in\Gamma^{(3)}$, then
  \begin{align*}
    \iota((r(\gamma_{0}),\varphi(\gamma_{1},\gamma_{2} )
    +\varphi(\gamma_{0},\gamma_{1}\gamma_{2})) 
    & =\tau(\gamma_{0})\tau(\gamma_{1})
      \tau(\gamma_{2})\tau(\gamma_{0}\gamma_{1}\gamma_{2})^{-1}\\ 
    & =\iota(r(\gamma_{0}),\varphi(\gamma_{0}\gamma_{1},\gamma_{2})+
      \varphi(\gamma_{0},\gamma_{1})). 
  \end{align*}
  Hence $\varphi$ is a 2-cocycle because $\iota$ is
  injective. Moreover $\varphi$ is normalized since $\tau(u)=u$ for
  all $u\in\go$. The map $\psi : \gpdtwst{\Gamma}{\varphi} \to\Sigma$
  defined by
  $\psi(g,\gamma) := g\cdot\tau(\gamma) =
  \iota(r(\gamma),g)\tau(\gamma)$ is a homeomorphism and a groupoid
  morphism.
\end{proof}

\bibliographystyle{amsxport} \bibliography{iksw}
\end{document}